\documentclass{amsart}
\usepackage{amssymb}
\usepackage{amsfonts}
\usepackage{latexsym}
\usepackage{extpfeil}
\usepackage{mathrsfs}
\usepackage{appendix}

\newtheorem{theorem}{Theorem}[section]
\newtheorem{lemma}[theorem]{Lemma}
\newtheorem{proposition}[theorem]{Proposition}
\newtheorem{corollary}[theorem]{Corollary}
\theoremstyle{definition}

\newtheorem{example}[theorem]{Example}

\newtheorem{remark}[theorem]{Remark}

\newcommand{\id}{{\rm id}}

\newcommand{\End}{\text{End}}

\newcommand{\Irr}{\text{\rm Irr}}
\newcommand{\FPdim}{\text{\rm FPdim}}
\newcommand{\Ext}{\text{\rm Ext}}

\newcommand{\Hom}{\text{Hom}}

\newcommand{\gr}{{\rm gr}}

\newcommand{\Rep}{{\rm Rep}}

\newcommand{\op}{{\text{op}}}

\newcommand{\C}{\mathcal{C}}
\newcommand{\D}{\mathcal{D}}
\newcommand{\E}{\mathcal{E}}
\newcommand{\sVec}{{\rm sVec}}

\renewcommand{\O}{\mathcal{O}}

\newcommand{\g}{\mathfrak{g}}

\newcommand{\ot}{\otimes}

\newcommand{\ben}{\begin{enumerate}}
\newcommand{\een}{\end{enumerate}}

\newcommand{\Rad}{{\text{Rad}}}
\newcommand{\Vect}{\text{Vec}}
\newcommand{\Lie}{{\text{Lie}}}

\newcommand{\Pro}{{\mathrm{Pro}}}
\newcommand{\cC}{{\mathcal C}}
\newcommand{\cA}{{\mathcal A}}

\newcommand{\cP}{\mathcal{P}}
\newcommand{\cV}{\mathcal{V}}
\newcommand{\Vecc}{{\rm Vec}}
\newcommand{\Gr}{{\rm Gr}}

\numberwithin{equation}{section}

\begin{document}

\title[symmetric integral tensor categories with the Chevalley property]{Finite symmetric integral tensor categories with the Chevalley property\\ \vspace{0.5cm}
{\SMALL \rm{With an appendix by Kevin Coulembier and Pavel Etingof }}}

\author{Pavel Etingof}
\address{Department of Mathematics, Massachusetts Institute of Technology,
Cambridge, MA 02139, USA} \email{etingof@math.mit.edu}

\author{Shlomo Gelaki}
\address{Department of Mathematics, Technion-Israel Institute of
Technology, Haifa 32000, Israel} \email{gelaki@math.technion.ac.il}

\date{\today}

\keywords{Symmetric tensor categories, Chevalley property, quasi-Hopf algebras, associators, Sweedler cohomology, finite group schemes}

\begin{abstract}
We prove that every finite symmetric integral tensor category $\mathcal{C}$ with the Chevalley property over an algebraically closed 
field $k$ of characteristic $p>2$ admits a symmetric fiber functor to the category of supervector spaces. This proves Ostrik's conjecture \cite[Conjecture 1.3]{o} in this case. Equivalently, we prove that there exists a unique finite supergroup scheme $\mathcal{G}$ over $k$ and a grouplike element $\epsilon\in k\mathcal{G}$ of order $\le 2$, whose action by conjugation on $\mathcal{G}$ coincides with the parity automorphism of $\mathcal{G}$, 
such that $\mathcal{C}$ is symmetric tensor equivalent to $\Rep(\mathcal{G},\epsilon)$. In particular, when $\mathcal{C}$ is unipotent, the functor lands in $\Vect$, so $\mathcal{C}$ is symmetric tensor equivalent to $\Rep(U)$ for a unique finite unipotent group scheme $U$ over $k$. We apply our result and the results of \cite{g} to classify certain finite dimensional triangular Hopf algebras with the Chevalley property over $k$ (e.g., local), in group scheme-theoretical terms. Finally, we compute the Sweedler cohomology of restricted enveloping algebras  over an algebraically closed field $k$ of characteristic $p>0$, classify associators for their duals, and study finite dimensional (not necessarily triangular) local quasi-Hopf algebras and finite (not necessarily symmetric) unipotent tensor categories over an algebraically closed field $k$ of characteristic $p>0$. 

The appendix by K. Coulembier and P. Etingof gives another proof of the above classification results using the recent paper \cite{Co}, and, more generally, shows that the maximal Tannakian and super-Tannakian subcategory of a symmetric tensor category over a field of characteristic $\ne 2$ is always a Serre subcategory. 
\end{abstract}

\maketitle

\section{Introduction}

This paper is motivated by the problem of classifying finite symmetric tensor categories $\mathcal{C}$ over an algebraically closed field $k$ of characteristic $p>0$. This problem was recently solved in the semisimple case by Ostrik \cite{o}, who proved that any symmetric fusion category $\mathcal{C}$ over $k$ admits a symmetric fiber functor to the Verlinde category $\text{Ver}_p$. In other words, Ostrik proved that $\mathcal{C}$ is symmetric tensor equivalent to the category $\Rep_0(G)$ of a unique finite group scheme $G$ in $\text{Ver}_p$ with a homomorphism $\pi_1(\mathcal{C})\to G$ such that the adjoint action of $\pi_1(\mathcal{C})$ on $\mathcal{O}(G)$ is canonical, where $\Rep_0(G)$ is the symmetric tensor category of representations of $G$ whose pullback to $\pi_1(\mathcal{C})$ is canonical.

The most natural family of non-semisimple finite symmetric 
tensor categories $\mathcal{C}$ over $k$ to start with is the
one consisting of such categories $\mathcal{C}$ in which the
Frobenius-Perron dimensions of objects are integers (i.e., $\mathcal{C}$ is {\em integral}). This is, in a sense, the easiest case, since by a result of \cite{eo}, such a category is the representation category of a finite dimensional triangular quasi-Hopf algebra over $k$. One is thus led naturally to the problem of classifying finite-dimensional triangular quasi-Hopf algebras over $k$. Like for Hopf algebras, the easiest non-semisimple finite dimensional triangular quasi-Hopf algebras to understand are those which have the {\em Chevalley property}. Recall that a quasi-Hopf algebra $H$ has the Chevalley property if the tensor product of every two simple $H$-modules is semisimple. For example, every basic (e.g., local) quasi-Hopf algebra has the Chevalley property. The objective of this paper is to classify finite dimensional triangular quasi-Hopf algebras over $k$ with the Chevalley property.

It is shown in \cite[Proposition 2.17]{eo} that a finite
dimensional local quasi-Hopf algebra over $\mathbb{C}$ is $1$-dimensional. On the other hand, there do exist non-trivial finite dimensional triangular local quasi-Hopf algebras over a field $k$ of characteristic $p>0$. For example, if $G$ is a finite unipotent group scheme over $k$ then its group algebra $kG$, equipped with the $R$-matrix $1\ot 1$, is a finite dimensional triangular local cocommutative Hopf algebra over $k$.

In categorical terms, the objective of this paper is to classify finite symmetric integral tensor categories $\mathcal{C}$ over $k$ which have the {\em Chevalley property}. Recall that a tensor category $\C$ has the Chevalley property if the tensor product of every two simple objects of $\C$ is semisimple. Namely, we prove the following theorem.

\begin{theorem}\label{classchptr}
Every finite symmetric integral tensor category $\mathcal{C}$ with the Chevalley property over an algebraically closed 
field $k$ of characteristic $p>2$ admits a symmetric fiber functor to the category $\text{sVec}$ of supervector spaces. Thus, there exists a unique finite supergroup scheme $\mathcal{G}$ over $k$ and a grouplike element $\epsilon\in k\mathcal{G}$ of order $\le 2$, whose action by conjugation on $\mathcal{G}$ coincides with the parity automorphism of $\mathcal{G}$, such that $\mathcal{C}$ is symmetric tensor equivalent to $\Rep(\mathcal{G},\epsilon)$.
\end{theorem}

\begin{remark}
\begin{enumerate}
\item
The proof of Theorem \ref{classchptr} uses \cite[Theorem 1.1]{o} and \cite[Theorem 8.1]{eov}, and the method we use in the proof is motivated by Drinfeld's paper \cite[Propositions 3.5, 3.6]{d}.
\item
Theorem \ref{classchptr} implies Ostrik's conjecture \cite[Conjecture 1.3]{o} for finite symmetric integral tensor categories $\mathcal{C}$ with the Chevalley property over an algebraically closed field $k$ of characteristic $p>2$. Namely, such categories $\mathcal{C}$ admit a symmetric fiber functor to $\text{sVec}\subset \text{Ver}_p$. However, there even exist fusion symmetric tensor categories over $k$ which are not integral, so in particular do not admit a symmetric fiber functor to $\text{sVec}$ (unlike in characteristic $0$ \cite{eg4}).
\end{enumerate}
\end{remark}

If $\mathcal{C}$ is {\em unipotent} (i.e., the unit object of $\mathcal{C}$ is its unique simple object) then it is integral and has the Chevalley property. Thus, specializing to unipotent symmetric tensor categories, we have the following result.
  
\begin{corollary}\label{classuniptr}
Every finite symmetric unipotent tensor category $\mathcal{C}$ over an algebraically closed field $k$ of characteristic $p>2$ admits a symmetric fiber functor to $\Vect$. Thus, there exists a unique finite unipotent group scheme $U$ over $k$ such that $\mathcal{C}$ is symmetric tensor equivalent to $\Rep(U)$.
\end{corollary}

\begin{remark}
\begin{enumerate}
\item
Corollary \ref{classuniptr} gives another proof of \cite[Proposition 3.6]{eho} for finite dimensional Hopf algebras.
\item
Unlike the proof of Theorem \ref{classchptr}, the proof of Corollary \ref{classuniptr} does not require \cite[Theorem 1.1]{o} and \cite[Theorem 8.1]{eov}.
\end{enumerate}
\end{remark}

The organization of the paper is as follows. In Section 2 we recall some necessary stuff about quasi-Hopf algebras, finite tensor categories, finite (super)group schemes, and cohomology. In particular, we compute the cohomology ring of the coordinate algebra of a finite connected group scheme in a functorial form, generalizing a result from \cite{fn}. Section 3 is devoted to the proof of Theorem \ref{classchptr} and Corollary \ref{classuniptr}. In particular, in Theorem \ref{symmwithphi} we classify 
finite dimensional triangular quasi-Hopf algebras with the Chevalley property over an algebraically closed field $k$ of characteristic $p>2$ (up to pseudotwist equivalence). Section 4 is devoted to the classification of certain finite dimensional triangular Hopf algebras with the Chevalley property (e.g., local) in characteristic $p>2$, in group scheme-theoretical terms (see Theorem \ref{classtrhas}). In Section 5 we compute the Sweedler cohomology groups \cite{s2} $H^{i}_{\text{Sw}}(u(\g))$, $i\ne 2$, of restricted enveloping algebras $u(\g)$ over an algebraically closed field $k$ of characteristic $p>0$ (see Theorem \ref{ncocy} and Corollary \ref{ncocymup}), and classify associators for their duals $u(\g)^*$ (see Corollary \ref{assocncocy} and Theorem \ref{assocalphap}). Finally, in Section 6 we consider finite dimensional (not necessarily triangular) local quasi-Hopf algebras and finite (not necessarily symmetric) unipotent tensor categories over an algebraically closed field $k$ of characteristic $p>0$ (see Theorems \ref{main2} and \ref{uniptc}).

The paper also contains an appendix by K. Coulembier and P. Etingof. This appendix gives another proof of Theorem \ref{classchptr} and Corollary \ref{classuniptr} using the recent paper \cite{Co}, and, more generally, shows that the maximal Tannakian and super-Tannakian subcategory of a symmetric tensor category over a field of characteristic $\ne 2$ is always a Serre subcategory. 

\begin{remark}
In a separate publication we plan to do the following:
\begin{enumerate}
\item 
Discuss the classification of finite symmetric integral tensor categories with the Chevalley property over an algebraically closed field of characteristic $2$.
\item
Compute the Sweedler cohomology group $H^{2}_{\text{Sw}}(u(\g))$ of restricted enveloping algebras $u(\g)$ over an algebraically closed field $k$ of characteristic $p>0$.
\item
Classify finite dimensional triangular Hopf algebras with the Chevalley property over an algebraically closed field $k$ of characteristic $p>0$.
\end{enumerate}
\end{remark}

{\bf Acknowledgments.} P. E. was supported by grant DMS-1502244. We are grateful to V. Ostrik for useful discussions. S. G. is grateful to the University of Michigan and MIT for their warm hospitality.

\section{preliminaries}

All constructions in this paper are done over an algebraically closed field $k$ of characteristic $p>0$.

\subsection{Quasi-Hopf algebras} Recall \cite{d}
that a {\em quasi-Hopf algebra} over $k$ is an algebra $H$ equipped with a
comultiplication $\Delta$ which is associative up to conjugation
by an invertible element $\Phi\in H^{\ot 3}$ (called the {\em
associator}), a counit $\varepsilon$, and an antipode $S$ together
with two distinguished elements $\alpha,\beta\in H$, satisfying
certain axioms. In particular, we have
\begin{equation}\label{phi1}
(\id\ot \id\ot \Delta)(\Phi)(\Delta\ot \id\ot \id)(\Phi) = (1\ot \Phi)(\id\ot \Delta\ot \id)(\Phi)(\Phi\ot 1),
\end{equation}

\begin{equation}\label{phieps}
(\varepsilon\ot \id\ot \id)(\Phi) = (\id\ot \varepsilon\ot \id)(\Phi)=(\id\ot \id\ot \varepsilon)(\Phi)=1
\end{equation} 
and
\begin{equation}\label{phicomm}
\Phi(\Delta\ot \id)(\Delta(h))\Phi^{-1}=(\id\ot \Delta)(\Delta(h)),\,\,h\in H.
\end{equation}

A {\em triangular} quasi-Hopf algebra is a quasi-Hopf algebra $H$ equipped with an invertible element $R\in H^{\ot 2}$, satisfying 
\begin{equation}\label{cocommR}
\Delta^{cop}(\cdot)=R\Delta(\cdot)R^{-1},
\end{equation}
\begin{equation}\label{triangR}
R^{-1}=R_{21}, 
\end{equation}
\begin{equation}\label{phi2}
(\Delta \ot \id)(R)=\Phi_{312}R_{13}\Phi_{132}^{-1}R_{23}\Phi_{123}
\end{equation}
and
\begin{equation}\label{phi3}
(\id\ot \Delta )(R)=\Phi_{231}^{-1}R_{13}\Phi_{213}R_{12}\Phi_{123}^{-1}.
\end{equation}
It is known that the pair $(R,\Phi)$ satisfies the following quantum Yang-Baxter equation:
\begin{equation}\label{qybe}
R_{12}\Phi_{312}R_{13}\Phi_{132}^{-1}R_{23}\Phi_{123}=
\Phi_{321}R_{23}\Phi_{231}^{-1}R_{13}\Phi_{213}R_{12}.
\end{equation}

An invertible element $J\in H\ot
H$ satisfying $(\varepsilon\ot \id)(J)=(\id\ot\varepsilon)(J)=1$ is
called a {\em pseudotwist} for $H$. Using a pseudotwist $J$ for $H$, one can
define a new quasi-Hopf algebra structure
$H^J=(H,\Delta^J,\varepsilon,\Phi^J,S^J,\beta^J,\alpha^J,1)$ on the
algebra $H$. In particular, 
\begin{equation}\label{phi4}
\Delta^J(h)=J^{-1}\Delta(h)J,\;h\in H,
\end{equation}
and 
\begin{equation}\label{phi5}
\Phi^J:=(1\ot J^{-1})(\id\ot \Delta)(J^{-1})\Phi(\Delta\ot
\id)(J)(J\ot 1).
\end{equation}
Moreover, if $H$ has a triangular $R$-matrix $R$ then 
\begin{equation}\label{twistedR}
R^J:=J_{21}^{-1}RJ 
\end{equation}
is a triangular $R$-matrix for $H^J$.
 
We will denote the {\em pseudotwist equivalence class} of $\Phi$ by $[\Phi]$.  
We will also say that two
quasi-Hopf algebras $H$ and $K$ are {\em pseudotwist equivalent} if $K$
and $H^J$ are isomorphic as quasi-Hopf algebras for some pseudotwist $J$
for $H$. In this situation $\Rep(H)$ and $\Rep(K)$ are
tensor equivalent. If $H$ is pseudotwist equivalent to a Hopf algebra,
we say that the associator of $H$ is {\em trivial}.

A {\em twist} for a Hopf algebra $H$ is a pseudotwist $J$ satisfying
\begin{equation}\label{twist}
(\id\ot \Delta)(J)(1\ot J)=(\Delta\ot \id)(J)(J\ot 1).
\end{equation}

\subsection{Finite tensor categories} Recall that a finite rigid tensor category ${\mathcal C}$ over
$k$ is a rigid tensor category over $k$ with
finitely many simple objects and enough projectives such that the
unit object ${\bf 1}$ is simple (see \cite{egno}). Let
$\Irr({\mathcal C})$ denote the set of isomorphism
classes of simple objects of ${\mathcal C}$. Then to each object
$X\in {\mathcal C}$ there is attached a non-negative number
$\FPdim(X)$, called the Frobenius-Perron (FP) dimension of $X$
(namely, it is the largest non-negative eigenvalue of the operator of
multiplication by $X$ in the Grothendieck ring of ${\mathcal C}$),
and the Frobenius-Perron (FP) dimension of ${\mathcal C}$ is
defined by
$$\FPdim({\mathcal C}):=\sum_{X\in \Irr({\mathcal C})}
\FPdim(X)\FPdim (P(X)),$$ where $P(X)$ denotes the projective cover
of $X$ (see \cite{egno}). For example, if ${\mathcal C}\cong \Rep(H)$, $H$ a finite dimensional quasi-Hopf algebra, then
$\FPdim(X)=\dim(X)$ and $\FPdim({\mathcal C})=\dim(H)$; in
particular the FP-dimensions of objects are integers. Moreover, by
\cite[Proposition 2.6 ]{eo}, $\mathcal{C}$ is equivalent to
$\Rep(H)$, $H$ a finite dimensional quasi-Hopf algebra, if and
only if $\mathcal{C}$ is integral (i.e., the FP-dimensions of its objects are integers).

\subsection{Finite dimensional algebras}\label{sta} Let $H$ be a finite dimensional algebra over $k$,
and let $I:=\Rad(H)$ be its Jacobson radical. Then we have
a filtration on $H$ by powers of $I$, so one can consider the
associated graded algebra $\gr(H)=\bigoplus_{r\ge 0} H[r]$,
$H[r]:=I^r/I^{r+1}$ ($I^0:=H$). By a standard result (see, e.g., \cite[Lemma 2.2]{eg}), $\gr(H)$ is generated by $H[0]$ and $H[1]$. If moreover $H$ is a quasi-Hopf algebra and $I$ is a quasi-Hopf ideal of $H$ (i.e., $H$ has the Chevalley property), then the associated graded algebra $\gr(H)$ has a natural structure of a graded  quasi-Hopf algebra.

We will need the following lemma.

\begin{lemma}\label{sqroot}
Let $H$ be a finite dimensional algebra over $k$, and let $I:=\Rad(H)$. Then the following hold:
\begin{enumerate}
\item
If $e\in H/I$ is an idempotent, then $e$ has a unique lift to an idempotent $\widetilde{e}\in H$ (up to conjugacy by elements of $1+I$). 
\item
Assume $p>2$. 
Then for every $h\in I$, there exists a unique element $a\in 1+I$ such that $a^2=1+h$.
\end{enumerate}
\end{lemma}

\begin{proof}
(1) See, e.g., \cite[Proposition 9.1.1]{irt}.

(2) Since $h$ is nilpotent, $h^{i}=0$ for some $i$. Also, for every integer $i$, the number of times $p$ appears in the numerator of $${{1/2}\choose{i}}=\frac{\frac{1}{2}(\frac{1}{2}-1)\cdots (\frac{1}{2}-i+1)}{i!}$$ equals at least the number of times it appears in the denominator. Hence, we see that $a:=\sum_{i\ge 0}{{1/2}\choose{i}}h^i$ is well defined, and it is clear that $a^2=1+h$. 
\footnote{Alternatively, let $\frac{1}{2}=a_0+a_1p+a_2p^2+\cdots$ be the $p$-adic representation of the $p$-adic integer $1/2$. Then $a=(1+h)^{a_0}(1+h^p)^{a_1}(1+h^{p^2})^{a_2}\cdots$.}

The uniqueness of $a$ follows by an easy induction. Namely, suppose $a$ is unique modulo $I^n$ and we have two solutions $a,a'$ modulo $I^{n+1}$. Then $a'=a+b$ for some $b\in I^n$ and $(a+b)^2=a^2$, so $ab+ba+b^2=0$. Hence, $2b=0$ modulo $I^{n+1}$, so $b=0$ modulo $I^{n+1}$, i.e., $a'=a$ modulo $I^{n+1}$, as desired.
\end{proof}

For every $h\in I$, the unique element $a$ from Lemma \ref{sqroot} is denoted by $(1+h)^{1/2}$.

\subsection{Finite group schemes}\label{fgsc} An affine group scheme $G$ over $k$ is called \emph{finite} if its representing Hopf algebra (= coordinate Hopf algebra) 
$\mathcal{O}(G)$ is finite dimensional. In this case,
$kG:=\mathcal{O}(G)^*$ is a finite dimensional cocommutative Hopf
algebra, which is called the \emph{group algebra} of $G$. A finite group scheme $G$ is called
\emph{constant} if its representing Hopf algebra $\mathcal{O}(G)$ is the Hopf algebra of functions on some finite
abstract group with values in $k$, and is called \emph{\'{e}tale} if $\mathcal{O}(G)$ is semisimple. Since $k$ is algebraically closed, it is known that $G$ is \'{e}tale if and only if it is a constant group scheme \cite[6.4]{w}. 

Let $G$ be a finite \emph{commutative} group scheme over $k$, i.e.,
$\mathcal{O}(G)$ is a finite dimensional commutative and
cocommutative Hopf algebra. In this case, $kG$ is also a finite dimensional commutative and cocommutative Hopf algebra, so it represents a finite commutative group scheme $G^D$ over $k$, which is called the \emph{Cartier dual} of $G$. For example, the Cartier dual of the group scheme $\alpha_p:=\mathbb{G}_{a,1}$ (= the Frobenius kernel of the additive group $\mathbb{G}_a$) is
$\alpha_p$, while the Cartier dual of $\mu_p:=\mathbb{G}_{m,1}$ (= the Frobenius kernel of the multiplicative group $\mathbb{G}_m$) is $\mathbb{Z}/p\mathbb{Z}$.

A finite group scheme $G$ is called \emph{connected} if $\mathcal{O}(G)$ is a local Hopf algebra. Equivalently, $G$ is connected if $\Rep(\mathcal{O}(G))$ is a finite unipotent tensor category over $k$.

A finite group scheme $U$ is called \emph{unipotent} if $\mathcal{O}(U)$ is a connected Hopf algebra (i.e., $\mathcal{O}(U)$ has a unique grouplike element). Equivalently, $U$ is unipotent if its group algebra $kU=\mathcal{O}(U)^*$ is a local cocommutative Hopf algebra. Thus, $\Rep(U)$ is a finite unipotent tensor category over $k$. Every finite \emph{commutative} unipotent group scheme $U$ over $k$ decomposes into a direct product $U=U^{\circ} \times P$, where $U^{\circ}$ is a connected commutative unipotent group scheme (i.e., the identity component of $U$) and $P$ is a constant $p$-group (see, e.g., \cite[6.8, p.52]{w}).

\subsection{Finite supergroup schemes}\label{fsgsc} Assume $p>2$. Recall that a finite {\em supergroup scheme} $\mathcal{G}$ over $k$ is a finite group scheme in the category $\text{sVec}$ of supervector spaces. Let $\mathcal{O}(\mathcal{G})$ be the coordinate Hopf superalgebra of $\mathcal{G}$ (it is supercommutative), and let $\mathcal{G}_0$ be the even part of $\mathcal{G}$. Then $\mathcal{G}_0$ is an ordinary finite group scheme over $k$. Let ${\bf \g}:=\Lie(\mathcal{G})$ be the Lie superalgebra of $\mathcal{G}$, and write ${\bf \g}={\bf \g}_0\oplus {\bf \g}_1$. We have ${\bf \g}_0=\Lie(\mathcal{G}_0)$ is the Lie algebra of $\mathcal{G}_0$. It is known that $\mathcal{O}(\mathcal{G})=\mathcal{O}(\mathcal{G}_0)\otimes \wedge {\bf \g}_1^*$ as superalgebras (see, e.g., \cite[Section 6.1]{ma}).

Let $\mathcal{G}$ be a finite supergroup scheme over $k$, let $\epsilon\in \mathcal{G}(k)$ be a grouplike element of $k\mathcal{G}$ of order $\le 2$ acting on $\mathcal{O}(\mathcal{G})$ by the parity automorphism, and let $$R_{\epsilon}:=\frac{1}{2}(1\otimes 1+1\otimes \epsilon + \epsilon \otimes 1 - \epsilon\otimes\epsilon).$$ Recall that $\Rep(\mathcal{G},\epsilon)$ is the category of representations of $\mathcal{G}$ on finite dimensional supervector spaces on which $\epsilon$ acts by parity, equipped with its standard tensor structure and the symmetric structure determined by $R_{\epsilon}$.

\subsection{Cohomology of products of $\alpha_{p,r}$'s}\label{cohom} 
Let $\alpha_{p,r}$ be the {\em $r$th Frobenius kernel of $\mathbb{G}_a$}. Then $k[x]/(x^{p^{r}})$, where $x$ is a primitive element, is the coordinate Hopf algebra of $\alpha_{p,r}$.

Let $G:=\prod_{i=1}^n\alpha_{p,r_i}$ and let $A:=\mathcal{O}(G)=k[x_1,\dots,x_n]/(x_1^{p^{r_1}},\dots,x_n^{p^{r_n}})$ be its coordinate Hopf algebra. Then $\mathfrak{a}:=H^1(G,k)=P(A)$ is the $(\sum_{i=1}^n r_i)$-dimensional vector space of primitive elements in $A$ with basis $\left\{x_j^{p^{i}}\mid 1\le j\le n,\,\,0\le i\le r_j-1\right\}$.
Let $$\xi(f):=\sum_{i=1}^{p-1}\frac{1}{i}\binom{p-1}{i-1}f^i\ot f^{p-i},\,\,\,f\in A.$$
Then $\xi$ defines a twisted\footnote{Namely, $\xi(af)=a^pf$ for every $a\in k$ and $f\in A$.} linear injective homomorphism $$\overline{\xi}:H^1(G,k)\to H^2(G,k)$$ 
(see, e.g., \cite[II, \textsection 3, 4.6]{dg}). Let $\mathfrak{a}^{(1)}$ be the Frobenius twist of $\mathfrak{a}$. Identifying $\mathfrak{a}^{(1)}$ with $\text{Im}(\overline{\xi})$, and writing $S'\mathfrak{a}^{(1)}$ for $S\mathfrak{a}^{(1)}$, where we put every element of $S^i\mathfrak{a}^{(1)}$ in degree $2i$, we have
$$H^{\bullet}(G,k)=\wedge \mathfrak{a} \ot S' \mathfrak{a}^{(1)},\,\,\,p>2,$$
and $$H^{\bullet}(G,k)=S \mathfrak{a},\,\,\,p=2.$$
(For more details see, e.g., \cite[Chapter 4]{j}.)

In particular, we have the following proposition. 
\begin{proposition}\label{h3233}
\begin{equation}\label{h33}
H^3(G,k)\cong \wedge^3 \mathfrak{a} \oplus (\mathfrak{a}\ot \mathfrak{a}^{(1)}),\,\,\,p>2,
\end{equation}
and 
\begin{equation}\label{h32}
H^3(G,k)\cong S^3 \mathfrak{a},\,\,\,p=2.
\end{equation}
\qed
\end{proposition}
We will use Proposition \ref{h3233} in Example \ref{assocalphap} below. 

\subsection{Cohomology of products of $\alpha_{p,r}^D$'s}\label{cohomd}
Let $A:=\mathcal{O}(\alpha_{p,r})=k[x]/(x^{p^{r}})$ be the coordinate Hopf algebra of $\alpha_{p,r}$. Let $$\{x^{i*}\mid 0\le i\le p^{r}-1\}$$ be the basis of $A^*=\mathcal{O}(\alpha_{p,r}^D)$ dual to the basis $\{x^i\mid 0\le i\le p^{r}-1\}$ of $A$. We have $A^*=
\bigotimes_{i=0}^{r-1}\left(k[x^{p^i*}]/(x^{p^i*})^p\right)$ as an algebra, with comultiplication map determined by 
\begin{equation}\label{comultp}
\Delta(x^{i*})=\sum_{j=0}^i x^{j*}\ot x^{(i-j)*}
\end{equation} 
for every $0\le i\le p^{r}-1$. In particular, the group $\mathbf{G}(A^*)$ of grouplike elements of $A^*$ is $1^*$, and $\mathfrak{b}:=P(A^*)$ is spanned by $x^*$.

Set
$$\beta:=\sum_{j=1}^{p^r-1}x^{j*}\ot x^{(p^r-j)*}\,\,\,\text{and}\,\,\,\gamma:=x^*\ot \beta.$$ 
It is straightforward to verify that
$\beta$ and $\gamma$ are Hochschild algebra $2$-cocycle and $3$-cocycle of $A$, respectively. Also, since $\beta\in A^*\ot A^*$ has degree $p^r$ it follows that $\beta$, and hence also $\gamma$, are non-trivial cocycles.

Let $H^{\bullet}(A,k)$ be the Hochschild cohomology of the algebra $A$ with coefficients in the trivial $A$-bimodule $k$. Note that by definition, we have $H^{\bullet}(A,k)=H^{\bullet}(\alpha_{p,r}^D,k)$. It is known that $H^1(A,k)=\text{Der}_k(A,k)$ 
(see, e.g., \cite[Lemma 9.2.1]{wi}), hence $H^1(A,k)=\mathfrak{b}$ is $1$-dimensional. It is also known that $H^{2i+1}(A,k)\cong H^1(A,k)$ and $H^{2i+2}(A,k)\cong H^2(A,k)$ for every $i\ge 0$. This can be seen by 
writing a $2$-periodic resolution of $k$ by free $A$-modules of rank $1$, with even differentials being multiplication by $x$ and odd ones by $x^{p^r-1}$ (see, e.g., \cite[Exsercise 9.1.4]{wi}). In particular, $H^2(A,k)$ and $H^3(A,k)$ are $1$-dimensional, spanned by $[\beta]$ and $[\gamma]$, respectively. (See also \cite[Remark 3.6]{fn}.) Furthermore, 
we have the following result (for a proof see, e.g., \cite[Proposition 3.5]{fn}).
\newpage

\begin{proposition}\label{hochcoh}
The following hold: 
\begin{enumerate}
\item 
If $p>2$, or $p=2$ and $r>1$, then $H^{\bullet}(\mathcal{O}(\alpha_{p,r}),k)=H^{\bullet}(\alpha_{p,r}^D,k)$ is a free supercommutative algebra generated by $[x^*]$ of degree $1$ (so, $[x^*]^2=0$) and $[\beta]$ of degree $2$. In particular, $H^{3}(\mathcal{O}(\alpha_{p,r}),k)=H^3(\alpha_{p,r}^D,k)$ is spanned by $[\gamma]=[\beta][x^*]$. 
\item
If $p=2$ then $H^{\bullet}(\mathcal{O}(\alpha_{p}),k)=H^{\bullet}(\alpha_{p}^D,k)$ is a free commutative algebra generated by $[x^*]$ of degree $1$. In this case, we have $[\beta]=[x^*]^2\ne 0$ and $[\gamma]=[\beta][x^*]=[x^*]^3$.
 \hspace{7cm} \qed
\end{enumerate}
\end{proposition}

\subsection{Cohomology of coordinate algebras of finite connected group schemes.}
Next, let $G$ be any finite {\em connected} group scheme over $k$. Then we have an {\em algebra} isomorphism
\begin{equation}\label{conncohomiso}
\mathcal{O}(G)=k[x_1,\dots,x_n]/\left(x_1^{p^{r_1}},\dots,x_n^{p^{r_n}}\right)
\end{equation}
(see, e.g, \cite[14.4, p. 112]{w}). Thus, by the K\"{u}nneth formula, we have an isomorphism of {\em graded commutative algebras} 
\begin{equation}\label{conncohom}
H^{\bullet}(\mathcal{O}(G),k)\cong \bigotimes _{i=1}^n H^{\bullet}\left(k[x_i]/\left(x_i^{p^{r_i}}\right),k\right).
\end{equation}

We would now like to give a functorial (i.e., coordinate-independent) formulation of this statement. 
Namely, we have the following generalization of \cite[Proposition 3.5]{fn}.

Let $\g$ be the Lie algebra of $G$. Thus any $z\in \g$ defines a derivation of $\mathcal O(G)$ into its augmentation module. Let $\mathfrak{m}$ be the maximal ideal of $\mathcal O(G)$ and $r$ be the smallest integer such that $\mathfrak{m}^{p^r}=0$.
Let $J_i$ be the ideal of $f\in \mathcal{O}(G)$ such that $f^{p^{i}}=0$. Thus $\mathfrak{m}=J_r\supset J_{r-1}\supset\cdots\supset J_0=0$. Let $F^i\g$ be the Lie subalgebra of $z\in \g$ such that $z(J_i)=0$. Thus $\g=F^0\g\supset F^1\g\supset\cdots\supset F^r\g=0$. 
Let $\gr(\g)_i:=F^{i-1}\g/F^i\g$, $1\le i\le r$. 

Set $\widetilde{\g}:=((J_1/J_1\mathfrak m)^*)^{(1)}$, where the superscript here and below denotes the Frobenius twist. Note that $V_0:=J_1/J_1\mathfrak m$ has a 
descending filtration by subspaces 
$$
V_i:=(J_{i+1}/(J_{i+1}\mathfrak{m}+J_i))^{p^{i-1}},
$$
 i.e., $V_0\supset V_1\supset\cdots\supset V_r=0$. Thus $\widetilde \g$ has an ascending filtration 
 $$
 0=W_0\subset W_1\subset\cdots\subset W_r=\widetilde{\g},
 $$ 
 where $W_i:=(V_i^\perp)^{(1)}$. 
 Note that $W_i/W_{i-1}={\rm gr}(\g)_i^{(i)}$. So ${\rm gr}(\widetilde{\g})=\bigoplus_{i=1}^r {\rm gr}(\g)_i^{(i)}$. 
 Thus $\widetilde\g$ has the same composition factors as $\g$, except they are Frobenius twisted and occur in opposite order. 

Finally, let $\phi: \g\to \widetilde{\g}$ be the twisted linear map given by the composition 
$\g\twoheadrightarrow {\rm gr}(\g)_1={\rm gr}(\g)_1^{(1)}\hookrightarrow \widetilde{\g}$. 

\begin{proposition}\label{kunneth}
The following hold: 
\begin{enumerate}
\item
If $p>2$ then we have a canonical isomorphism of graded algebras 
$$
H^{\bullet}(\mathcal{O}(G),k)\cong\wedge \g\otimes S'(\widetilde{\g}).
$$ 
\item
If $p=2$ then we have a canonical isomorphism of graded algebras 
$$
H^{\bullet}(\mathcal{O}(G),k)\cong S\g\otimes S'(\widetilde\g)/(z^2-\phi(z), z\in \g).  
$$  
\end{enumerate} 
\end{proposition}

\begin{proof} Let us construct a canonical linear map $\eta: \widetilde\g\to H^2(\mathcal{O}(G),k)$.
To this end let $Y\subset G(k[t]/t^{p^r})$ be the set of points $y$ such that $y(J_i)\subset (t^{p^{r-i}})$ for all $i$. 
Thus each $y\in Y$ defines a linear map $c_{y}: J_1\to k$ sending $f\in J_1$ to the coefficient of $t^{p^{r-1}}$ in $f(y)=y(f)$. 
It is clear that $c_{y}(J_1\mathfrak{m})=0$, so $c_{y}$ descends to a linear map 
$J_1/J_1\mathfrak{m}\to k$. In other words, we can view $c_y$ as an element of $\widetilde{\g}$. 
Thus we obtain a canonical map $c: Y\to \widetilde{\g}$. It is easy to see using  the explicit presentation of $\mathcal O(G)$ given by \eqref{conncohomiso} that $c$ is surjective. 

For $y\in Y$, let us try to lift the homomorphism $y: \mathcal{O}(G)\to k[t]/t^{p^r}$ to a homomorphism 
$\mathcal O(G)\to k[t]/t^{p^r+1}$. This runs into an obstruction in $H^2(\mathcal O(G),k)$. Namely, let $\widetilde{y}$ be any lift of $y$ just as a linear map. Then $\widetilde{y}(ab)-\widetilde{y}(a)\widetilde{y}(b)=t^{p^r}\omega(a,b)$, where 
$\omega\in Z^2(\mathcal O(G),k)$, and the cohomology class of $\omega$ is independent of 
the choice of the lifting. Thus we obtain a well-defined map $\theta: Y\to H^2(\mathcal{O}(G),k)$, given by $\theta(y)=[\omega]$. 

It is easy to check using  \eqref{conncohomiso} that 
the map $\theta$ factors through $c$, i.e., there exists a (necessarily unique) 
map $\eta: \widetilde{\g}\to H^2(\mathcal O(G),k)$ such that $\theta=\eta\circ c$. Moreover, it is easy to see that $\eta$ is linear. Thus $\eta$ is a desired linear map $\widetilde{\g}\to H^2(\mathcal O(G),k)$. 

There is also an obvious canonical inclusion $\iota: \g\hookrightarrow H^1(\mathcal{O}(G),k)$. If $p>2$, the maps $\iota$ and $\eta$ define a graded algebra map $\xi: \wedge \g\otimes S'(\widetilde \g)\to H^\bullet(\mathcal{O}(G),k)$. It is easy to check using \eqref{conncohomiso} that $\xi$ is an isomorphism, which proves (1). 
If $p=2$, we similarly have a graded map $\hat\xi: S \g\otimes S'(\widetilde \g)\to H^\bullet(\mathcal O(G),k)$, which annihilates $z^2-\phi(z)$ for any $z\in \g$, so descends to a graded map $\xi: S\g\otimes S'(\widetilde\g)/(z^2-\phi(z), z\in \g)\to H^\bullet(\mathcal O(G),k)$, which is again checked to be an isomorphism using \eqref{conncohomiso}. This proves (2). 
\end{proof}

\begin{remark} 1. The group structure of $G$ is irrelevant in Proposition \ref{kunneth}; only its scheme structure (i.e., the algebra structure of $\mathcal O(G)$) matters. In other words, 
the description of the cohomology ring of $\mathcal{O}(G)$ in Proposition \ref{kunneth} is functorial 
even in the category of schemes, where there are more morphisms than in the category of group schemes. 

2. The filtration on $\widetilde{\g}$ does not split canonically, even in the category of group schemes. Indeed, take $G=\alpha_{p,1}\times \alpha_{p,2}$, i.e., 
$\mathcal{O}(G)=k[x,y]/(x^p,y^{p^2})$, where $x$ and $y$ are primitive. 
Then $J_1/J_1\mathfrak m$ is spanned by $x$ and $y^p$. 
We have an automorphism of $G$ given by $x\mapsto x+y^p,y\mapsto y$. 
It is clear that this automorphism acts unipotently with respect to the filtration but nontrivially on $J_1/J_1\mathfrak m$, which shows that the filtration on $J_1/J_1\mathfrak{m}$ and hence on $\widetilde{\g}$ does not split canonically. 

3. Note that if all $r_i=1$ then Proposition \ref{kunneth} reduces to (\ref{h33})-(\ref{h32}).
\end{remark}
 
Set $x_t^{l*}:=(x_t^l)^*$, $1\le t\le n$, $0\le l\le r_t-1$. For every $1\le i,j\le n$, let
\begin{equation}\label{betagamma}
\beta_j:=\sum_{l=1}^{p^{r_j}-1}x_j^{l*}\ot x_j^{(p^{r_j}-l)*}\,\,\,\text{and}\,\,\,\gamma_{ij}:=x_i^{*}\ot \beta_j.
\end{equation} 

\begin{corollary}\label{thirdhochcoh}
Let $G$ be a finite connected group scheme over $k$ with Lie algebra $\g$ as above. Then we have a short exact sequence
$$0\to \g \ot \widetilde{\g} \to H^{3}(\mathcal{O}(G),k)\to \wedge^3\g \to 0,$$ which canonically splits for $p>2$. Moreover, the set 
$$\{[\gamma_{ij}]\mid 1\le i,j\le n\}\cup \{[x_i^*\ot x_j^*\ot x_l^*]\mid 1\le i<j<l\le n\}
$$
forms a linear basis for $H^{3}(\mathcal{O}(G),k)$.
\end{corollary}

\begin{proof}
Follows from Proposition \ref{kunneth}.
\end{proof} 

\section{The proof of Theorem \ref{classchptr}}

In the next two sections assume $k$ has characteristic $p>2$, unless otherwise explicitly specified.

By \cite[Theorem 2.6]{eo}, $\mathcal{C}$ is symmetric tensor equivalent to $\Rep(H,R,\Phi)$ for some finite dimensional triangular quasi-Hopf algebra $(H,R,\Phi)$ over $k$ with the Chevalley property. Thus, we have to prove the following theorem.

\begin{theorem}\label{symmwithphi}
Let $(H,R,\Phi)$ be a finite dimensional triangular quasi-Hopf algebra over $k$ with the Chevalley property. Then there exists a pseudotwist $F$ for $H\ot H$ such that $$(H,R,\Phi)^F=(H^F,R_u,1),$$ where $u\in H^F$ is a grouplike element of order $\le 2$.
\end{theorem}

We will prove Theorem \ref{symmwithphi} in several steps.

\subsection{$\gr(H)$} Let $(H,R,\Phi)$ be a finite dimensional triangular quasi-Hopf algebra over $k$ with the {\em Chevalley property}. Let $I:=\Rad(H)$ be the Jacobson radical of $H$. Since $I$ is a quasi-Hopf ideal of $H$, the associated graded algebra $\gr(H)=\bigoplus_{r\ge 0} H[r]$ has a natural structure of a graded triangular quasi-Hopf algebra with some $R$-matrix $R_0\in H[0]^{\otimes 2}$ and associator $\Phi_0\in H[0]^{\otimes 3}$ (see \ref{sta}).

\begin{proposition}\label{main1}
The following hold:
\begin{enumerate}
\item
$H[0]$ is semisimple.
\item
$(H[0],R_0,\Phi_0)$ is a triangular quasi-Hopf subalgebra of $(\gr(H),R_0,\Phi_0)$.
\item
$\Rep(H[0],R_0,\Phi_0)$ is symmetric tensor equivalent to $\Rep(G,\epsilon)$ for some finite semisimple group scheme $G$ over $k$.
\item
$(\gr(H),R_0,\Phi_0)$ is pseudotwist equivalent to a graded triangular Hopf algebra with $R$-matrix $R_{\epsilon}$, whose degree $0$-component is $(kG,R_{\epsilon},1)$
\end{enumerate}
\end{proposition}

\begin{proof}
Parts (1) and (2) are clear. 

Since $\Rep(H[0],R_0,\Phi_0)$ is an integral symmetric fusion category, Part (3) follows from \cite[Theorem 8.1]{eov} (which in turn relies on \cite[Theorem 1.1]{o}).

By Part (3), there exists a pseudotwist $J$ for $H[0]$ such that $$(H[0],R_0,\Phi_0)^J\cong(kG,R_{\epsilon},1).$$ Hence, 
pseudotwisting $\gr(H)$ using $J$, we have $$(\gr(H),R_0,\Phi_0)^J\cong(\gr(H)^J,R_{\epsilon},1),$$ 
and since $J$ has degree $0$, we have $\gr(H)^J=\left(\bigoplus_{r\ge 0} H[r]\right)^J$, as claimed in Part (4).
\end{proof} 

\begin{corollary}\label{gr}
Let $(H,R,\Phi)$ be a finite dimensional triangular quasi-Hopf algebra over $k$ with the Chevalley property. 
Then $\gr(H)$ is pseudotwist equivalent to $k\mathcal{G}$ for some finite supergroup scheme $\mathcal{G}$ over $k$ containing $G$ as a closed subgroup scheme. \qed
\end{corollary}

\begin{remark}\label{epsilon=1}
(1) By Nagata's theorem (see \cite[p.223]{a}), we have $G=\Gamma\ltimes P^D$, where $\Gamma$ is a finite group of order coprime to $p$ and $P$ is a finite abelian $p$-group. Hence, we have $kG=k\Gamma\ltimes \mathcal{O}(P)$. Thus, $\epsilon\in \Gamma$ is a central element of order $\le 2$ acting trivially on $\mathcal{O}(P)$.

(2) Note that $\mathcal{G}$ is an ordinary group scheme if and only if $\epsilon$ is central in $\gr(H)$. Hence, if $\epsilon=1$ then $\mathcal{G}$ is an ordinary group scheme.
\end{remark}

\subsection{Trivializing $R$} 
By Corollary \ref{gr}, we can assume that $\gr(H)=k\mathcal{G}$ as {\em superalgebras}.

Pick a lift $u$ of $\epsilon$ to $H$ such that $u^2=1$. Note that by Lemma \ref{sqroot}(1), $u$ is unique up to conjugation by $1+\Rad(H)$. Set
$$
R_u:=\frac{1}{2}(1\ot 1+1\ot u+u\ot 1-u\ot u).$$

\begin{proposition}\label{killingR} 
There exists a pseudotwist $J$ for $H^{\ot 2}$, projecting to $1$ in $H[0]^{\ot 2}$, such that $$(H,R,\Phi)^J=(H^J,R_u,\Phi^J),$$ and $J^{-1}\Delta(u)J=u\otimes u$. 
\end{proposition}

\begin{proof}
We need to show that there exists a pseudotwist $J$ in $H^{\ot 2}$, projecting to $1$ in $H[0]^{\ot 2}$, such that $J_{21}^{-1}RJ=R_u$ and $J^{-1}\Delta(u)J=u\otimes u$. Let us first find $J$ satisfying the first condition. This condition is equivalent to the identity $J^{-1}\sigma RJ=\sigma R_u$ in $k[\mathbb{Z}/2\mathbb{Z}]\ltimes H^{\ot 2}$, where $\sigma$ is the generator of $\mathbb{Z}/2\mathbb{Z}$.

Applying Lemma \ref{sqroot}(1) to $A:=k[\mathbb{Z}/2\mathbb{Z}]\ltimes H^{\ot 2}$ and $e:=(\sigma R_{\epsilon} +1\otimes 1)/2$, we get that $\widetilde{e}_1:=(\sigma R +1\otimes 1)/2$ and $\widetilde{e}_2:=(\sigma R_u +1\otimes 1)/2$ are two liftings of $e$, so there exists $L\in A$, projecting to $1$ in $k[\mathbb{Z}/2\mathbb{Z}]\ltimes H[0]^{\ot 2}$, such that $\widetilde{e}_1L=L\widetilde{e}_2$, hence $\sigma RL=L\sigma R_u$. Writing $L=J+\sigma K$, where $J,K\in H^{\ot 2}$, we find that $\sigma RJ=J\sigma R_u$ and $J$ projects to $1$ in $H[0]^{\ot 2}$, as desired.

Thus, we may (and will) assume without loss of generality that $R=R_u$. It remains to find a pseudotwist $J$ projecting to $1$ in $H[0]^{\otimes 2}$ such that $J$ commutes with $\sigma R_u$ and $J^{-1}\Delta(u)J=u\otimes u$. To this end, note that $\widetilde e_1:=(\Delta(u)+1\otimes 1)/2$ and $\widetilde e_2:=(u\otimes u+1\otimes 1)/2$ are both liftings of the idempotent 
$e:=(\epsilon\ot \epsilon+1\ot 1)/2$. Thus Lemma \ref{sqroot}(1) applied to the centralizer 
$A'$ of $\sigma R_u$ in $A$ yields an element $L\in A$ projecting to $1$ in $k[\mathbb{Z}/2\mathbb{Z}]\ltimes H[0]^{\ot 2}$ and commuting with $\sigma R_u$ 
such that $\Delta(u)L=L(u\otimes u)$. Writing $L=J+\sigma K$, where $J,K\in H^{\ot 2}$, we find that $\Delta(u)J=J(u\otimes u)$ and $J$ commutes with $\sigma R_u$ and projects to $1$ in $H[0]^{\ot 2}$, as desired.
\end{proof}

\begin{remark}
If $\epsilon=1$ then $u=1$, and we may take $J:=R^{-1/2}=R_{21}^{1/2}$ (which is a well-defined invertible element in $H\ot H$ by Lemma \ref{sqroot}(2)). Indeed, it is straightforward to verify that we have $$J_{21}^{-1}RJ=R_{21}^{1/2}RR_{21}^{1/2}=R_{21}R=1.$$
\end{remark}

\begin{remark}\label{venk}
By \cite[Section 1.5]{v}, the assumption that $p\ne 2$ in Proposition \ref{killingR} is essential, even when $\Phi=1$. Namely, it is shown there that the tensor category $\Rep(\alpha_2)$, equipped with the symmetric structure determined by the $R$-matrix $R:=1\ot 1 + x\ot x$, does not admit a symmetric fiber functor to $\text{Ver}_2=\Vect$.
\end{remark}

\subsection{Trivializing $\Phi$} Let $(H,R_u,\Phi)$ be a finite dimensional triangular quasi-Hopf algebra over $k$ with the Chevalley property, where $u\in H$ satisfies $\Delta(u)=u\otimes u$, $u^2=1$ and projects to $\epsilon$ in $H[0]$. Let $m$, $\varepsilon$ denote the multiplication and counit maps of $\mathcal{O}(\mathcal{G})$.

If $\Phi\ne 1$, consider $\Phi-1$. If it has degree $r$ then let $\phi$ be its projection to $\gr(H)^{\ot 3}[r]$. 
Then $\phi$ is even. 

For every permutation $(i_1i_2i_3)$ of $(123)$, we will use $\phi_{i_1i_2i_3}$ to denote the $3$-tensor obtained by permuting the components of $\phi$ accordingly, multiplied by $\pm 1$ according to the sign rule.

\begin{lemma}\label{qtalt} 
The following hold:

\begin{enumerate}
\item  
$$\phi\circ (\id\ot \id\ot m)+\phi\circ (m\ot \id\ot \id)= \varepsilon\ot \phi +\phi\circ (\id\ot m\ot \id)+ \phi\ot \varepsilon$$
and
$$
\phi\circ (\id\ot \id\ot 1)=\phi\circ (1\ot \id\ot \id)=\phi\circ (\id\ot 1\ot \id)=\varepsilon \ot \varepsilon,$$ 
i.e., $\phi\in Z^3(\mathcal{O}(\mathcal{G}),k)$ is an {\em even}  normalized Hochschild $3$-cocycle of $\mathcal{O}(\mathcal{G})$ with coefficients in the trivial module $k$.
\item 
${\rm Alt}(\phi):=\phi_{312}-\phi_{132}+\phi_{123} +\phi_{231}-\phi_{213}-\phi_{321}=0$.

\item
$\phi_{123}+\phi_{321}=0$.
\end{enumerate}
\end{lemma}

\begin{proof}
(1) Follows from (\ref{phi1}) and (\ref{phieps}) in a straightforward manner.

(2) Follows from (\ref{qybe}) in a straightforward manner.

(3) Using (\ref{phi2}), (\ref{phi3}) and the identity $\Delta(u)=u\otimes u$, it is straightforward to verify that 
\begin{equation}\label{eq1}
\phi_{312}-\phi_{132}+\phi_{123}=0
\end{equation}
and
\begin{equation}\label{eq2}
-\phi_{231}+\phi_{213} -\phi_{123}=0.
\end{equation}
Therefore, by using Part (2), we get
\begin{eqnarray*}
\lefteqn{0=\phi_{312}-\phi_{132}+\phi_{123}-(-\phi_{231}+\phi_{213} -\phi_{123})}\\
& = & {\rm Alt}(\phi)+\phi_{321}+\phi_{123}\\
& = & \phi_{123}+\phi_{321},
\end{eqnarray*}
as claimed.
\end{proof}

Let $\mathcal{G}_0$ be the even part of $\mathcal{G}$, and let ${\bf \g}={\bf \g}_0\oplus {\bf \g}_1$ be the Lie superalgebra of $\mathcal{G}$, where ${\bf \g}_0=\Lie(\mathcal{G}_0)$ is the Lie algebra of $\mathcal{G}_0$ (see \ref{fsgsc}). Since 
$\mathcal{O}(\mathcal{G})=\mathcal{O}(\mathcal{G}_0)\otimes \wedge {\bf \g}_1^*$ as superalgebras, it follows from the  K\"{u}nneth formula that $$H^{\bullet}(\mathcal{O}(\mathcal{G}),k)=H^{\bullet}(\mathcal{O}(\mathcal{G}_0^{\circ}),k)\otimes H^{\bullet}(\wedge {\bf \g}_1^*,k)=H^{\bullet}(\mathcal{O}(\mathcal{G}_0^{\circ}),k)\otimes S^{\bullet}{\bf \g}_1,$$ 
where $\mathcal{G}_0^{\circ}$ is the identity component of $\mathcal{G}_0$. In particular, we have
\begin{eqnarray*}
\lefteqn{H^{3}(\mathcal{O}(\mathcal{G}),k)}\\
& = &
H^{3}(\mathcal{O}(\mathcal{G}_0^{\circ}),k)\oplus \left(H^{2}(\mathcal{O}(\mathcal{G}_0^{\circ}),k)\otimes {\bf \g}_1\right)\oplus \left(H^{1}(\mathcal{O}(\mathcal{G}_0^{\circ}),k)\otimes S^{2}{\bf \g}_1\right)\oplus S^{3}{\bf \g}_1\\
& = & H^{3}(\mathcal{O}(\mathcal{G}_0^{\circ}),k)\oplus \left(H^{2}(\mathcal{O}(\mathcal{G}_0^{\circ}),k)\otimes {\bf \g}_1\right)\oplus \left({\bf \g}_0\otimes S^{2}{\bf \g}_1\right)\oplus S^{3}{\bf \g}_1.
\end{eqnarray*}
Thus by Corollary \ref{thirdhochcoh}, the even part $H^{3}(\mathcal{O}(\mathcal{G}),k)_{\text{even}}$ of $H^{3}(\mathcal{O}(\mathcal{G}),k)$ is given by
\begin{eqnarray*}
\lefteqn{H^{3}(\mathcal{O}(\mathcal{G}),k)_{\text{even}}}\\
& = &
H^{3}(\mathcal{O}(\mathcal{G}_0^{\circ}),k)\oplus \left({\bf \g}_0\otimes S^{2}{\bf \g}_1\right)\\
& = & 
\left(\g_0\otimes \widetilde{\g_0}\right)\oplus \wedge^3\g_0 \oplus \left(\g _0\otimes S^{2}\g _1\right).
\end{eqnarray*}

\begin{proposition}\label{killingPhi}
The $3$-cocycle $\phi$ is a coboundary.
\end{proposition}

\begin{proof} 
Let $\{\tau_j\}_{j=1}^m$ be a basis of $\g_1$, and retain the notation and linear basis introduced in (\ref{betagamma}). Then by Corollary \ref{thirdhochcoh}, Lemma \ref{qtalt} and the above, we can express $\phi\in Z^3(\mathcal{O}(\mathcal{G}),k)$ in the following form:
$$
\phi=\sum_{i,j=1}^{n}a_{ij}
\gamma_{ij}+\sum_{1\le i<j<l\le n}b_{ijl}(x_i^*\ot x_j^*\ot x_l^*)
+ \sum_{\substack{1\le i\le n \\ 1\le j\le k\le m}}c_{ijk}(x_i^*\ot \tau_j\ot \tau_k) + df,$$ 
for some $a_{ij},b_{ijl},c_{ijk}\in k$ and even  
$f\in k\mathcal{G}^{\ot 2}$. 
Since we have $[\gamma_{ij}]=[(\gamma_{ij})_{31}]$, $[x_i^*\ot x_j^*\ot x_l^*]=-[x_l^*\ot x_j^*\ot x_i^*]$ and $[x_i^*\ot \tau_j\ot \tau_k]=-[\tau_k\ot \tau_j\ot x_i^*]$, it follows that 
$$
\phi_{321}=\sum_{i,j=1}^{n}a_{ij}\gamma_{ij}
-\sum_{1\le i<j<l\le n}b_{ijl}(x_i^*\ot x_j^*\ot x_l^*)-\sum_{\substack{1\le i\le n \\ 1\le j\le k\le m}}c_{ijk}(x_i^*\ot \tau_j\ot \tau_k) +dg,
$$
for some $g\in k\mathcal{G}^{\ot 2}$.
Since $\phi_{123}+\phi_{321}=0$ by Lemma \ref{qtalt}, we have
$$
0=\phi_{123}+\phi_{321}=\sum_{i,j=1}^{n}2a_{ij}\gamma_{ij}
+d(f+g),
$$
which implies that $\sum_{i,j=1}^{n}2a_{ij}[\gamma_{ij}]=0$. But 
$p\ne 2$, hence $a_{ij}=0$ for every $i,j$.  
Thus, we have 
\begin{equation}\label{phiinp}
\phi=\sum_{1\le i<j<l\le n}b_{ijl}(x_i^*\ot x_j^*\ot x_l^*)+ \sum_{\substack{1\le i\le n \\ 1\le j\le k\le m}}c_{ijk}(x_i^*\ot \tau_j\ot \tau_k)+df.
\end{equation}

Now by Lemma \ref{qtalt}, $\text{Alt}(\phi)=0$, and since  $k\mathcal{G}$ is super cocommutative we have $\text{Alt}(df)=0$. Hence by (\ref{phiinp}), 
\begin{equation*}
0=\text{Alt}(\phi)=\sum_{1\le i<j<l\le n}b_{ijl}\text{Alt}(x_i^*\ot x_j^*\ot x_l^*)+\sum_{\substack{1\le i\le n \\ 1\le j\le k\le m}}c_{ijk}\text{Alt}(x_i^*\ot \tau_j\ot \tau_k).
\end{equation*}
Finally since the tensors $\text{Alt}(x_i^*\ot x_j^*\ot x_l^*)$, $\text{Alt}(x_i^*\ot \tau_j\ot \tau_k)$ are linearly independent, it follows that $b_{ijl}=c_{ijk}=0$ for every $i,j,l,k$. Thus $\phi=df$ is a coboundary, as claimed.
\end{proof}

\subsection{The proof of Theorem \ref{symmwithphi}} 
Let $\Gamma:=\mathcal{G}_0/\mathcal{G}_0^{\circ}$. Then $\Gamma$ is a finite constant group with 
order not divisible by $p$. Since we have $\mathcal{G}_0=\Gamma\ltimes \mathcal{G}_0^{\circ}$, it follows that $\mathcal{O}(\mathcal{G}_0)=\mathcal{O}(\Gamma)\otimes \mathcal{O}(\mathcal{G}_0^{\circ})$ as algebras.
 
By Proposition \ref{killingPhi}, we have 
$\phi=df$ for some even $f\in (\mathcal{O}(\mathcal{G})^*)^{\ot 2}$ with the same degree $r$ as $\phi$. 
Since $\phi_{123}+\phi_{321}=0$, it follows that $\phi=d(f_{21})$. Thus, replacing $f=f_{12}$ by $(f_{12}+f_{21})/2$ (this is possible as $p>2$), we see that $\phi=df$ for some  
even and symmetric (in the super sense) element $f\in (\mathcal{O}(\mathcal{G})^*)^{\ot 2}$ with the same degree as $\phi$, i.e., $f$ commutes with $\sigma R_{\epsilon}$ and $\epsilon\ot \epsilon$. 

Choose an even symmetric lift $\widetilde{f}$ of $f$ to $H^{\ot 2}$, i.e., $\widetilde{f}$ commutes with $\sigma R_{u}$ and $u\otimes u$ (it is clear that such a lift exists). Then it follows from the above that the pseudotwist $F:=1+\widetilde{f}$ is 
even and symmetric, so $(H,R_u,\Phi)^{F}=(H^{F},R_u,\Phi^{F})$, and the pseudotwisted associator $\Phi^{F}$ is equal to $1+$ terms of degree $\ge r+1$. Also, the condition $\Delta(u)=u\ot u$ is preserved. By continuing this procedure, we will come to a situation where $(H,R_u,\Phi)^F=(H^F,R_u,1)$ for some pseudotwist $F\in H^{\ot 2}$, as desired. \qed

Since local quasi-Hopf algebras have the Chevalley property, Theorem \ref{symmwithphi} and Remark \ref{epsilon=1} imply the following special case.

\begin{theorem}\label{localp>2}
Let $(H,R,\Phi)$ be a finite dimensional local triangular quasi-Hopf algebra over $k$. Then  
$(H,R,\Phi)$ is pseudotwist equivalent to a triangular cocommutative Hopf algebra with $R$-matrix $1\ot 1$. \qed
\end{theorem}

\begin{remark}
Note that Theorem \ref{localp>2} is equivalent to Corollary \ref{classuniptr}.
\end{remark}

The following corollary in known in the Hopf case even without the symmetry assumption (see Remark \ref{newrmk} below).

\begin{corollary}\label{alphapsym}
Let $\mathcal{C}$ be a finite symmetric unipotent tensor category over $k$ such that $\FPdim(\mathcal{C})=p$. Then $\mathcal{C}$ is symmetric tensor equivalent to either $\Rep(\mathbb{Z}/p\mathbb{Z})$ or $\Rep(\alpha_p)$.
\end{corollary}

\begin{proof}
Follows immediately from Theorem \ref{classuniptr}. (See also  Corollary \ref{cor3} below.)
\end{proof}

\section{triangular Hopf algebras with the Chevalley property}

We first observe the following consequence of Theorem \ref{symmwithphi}.

\begin{corollary}\label{symchevprop}
Let $(H,R)$ be a finite dimensional triangular Hopf algebra with the Chevalley property over $k$. Then there exists a finite supergroup scheme $\mathcal{G}$ over $k$ such that $(H,R)$ is {\em twist} equivalent to $(k\mathcal{G},R_{\epsilon})$. If moreover $H$ is local, then there exists a finite unipotent group scheme $U$ over $k$ such that $(H,R)$ is {\em twist} equivalent to $(kU,1\ot 1)$.
\end{corollary}

\begin{proof}
Applying Theorem \ref{symmwithphi} to $(H,R,1)$ yields the  existence of a pseudotwist $J$ for $H$ such that $(H,R,1)^J=(H^J,R_{\epsilon},1)$. In particular, we have $1^J=1$, which is equivalent to $J$ being a twist.
\end{proof}

Next we observe that Corollary \ref{symchevprop} together with \cite[Corollary 6.3 \& Proposition 6.7]{g} imply the classification of certain finite dimensional triangular Hopf algebras with the Chevalley property over $k$. 

More precisely, let $\mathcal{T}$ be the set of all finite dimensional triangular Hopf algebras $H$ with the Chevalley property over $k$, such that $\gr(H)$ is the group algebra of a finite group scheme over $k$ (see Corollary \ref{gr}). For example, every finite dimensional triangular local Hopf algebra belongs to $\mathcal{T}$. 

Let $L$ be any finite group scheme over $k$. 
Recall that a twist $J$ for $kL$ is called \emph{minimal} if the triangular Hopf algebra $(kL^J,J_{21}^{-1}J)$ is minimal \cite{r}, i.e., if the left (right) tensorands of $J_{21}^{-1}J$ span $kL$. 

Recall also that a $2$-cocycle $\psi:L\times L\to \mathbb{G}_m$ (equivalently, a twist $\psi$ for $\mathcal{O}(L)$) is called \emph{non-degenerate} if the category $\text{Corep}(\mathcal{O}(L)_{\psi})$ of finite dimensional comodules over $\mathcal{O}(L)_{\psi}$ is equivalent to $\Vect$ (i.e., the coalgebra $\mathcal{O}(L)_{\psi}$ obtained from $\mathcal{O}(L)$ by twisting its comultiplication on one side is simple).

We are now ready to state the following classification result.

\begin{theorem}\label{classtrhas}
The following three sets are in canonical bijection with each other:
\begin{enumerate}
\item
Isomorphism classes of triangular Hopf algebras $H$ in $\mathcal{T}$.
\item
Conjugacy classes of triples $(G,L,J)$, where $G$ is a finite group scheme over $k$, $L$
is a closed group subscheme of $G$, and $J$ is a minimal twist for $kL$.
\item 
Conjugacy classes of triples $(G,L,\psi)$, where $G$ is a finite group scheme over $k$, 
$L$ is a closed group subscheme of $G$, and $\psi$ 
is a non-degenerate $2$-cocycle on $L$ with coefficients in $\mathbb{G}_m$.
\end{enumerate}
\end{theorem}

\begin{remark}
The following hold:
\begin{enumerate}
\item
The correspondence between (1) and (2) in Theorem \ref{classtrhas} is given by $H=kG^J$. A non-degenerate $2$-cocycle $\psi$ on $L$ as in Theorem \ref{classtrhas}(3) determines a module category over $\Rep(G)$ of rank $1$, i.e., a tensor structure on the forgetful functor $\Rep(G)\to \Vect$, thus a twist $J$ for $kG$ supported on $L$. 
\item
Theorem \ref{classtrhas} is the analog of \cite[Theorem 5.1]{eg4} for finite dimensional triangular Hopf algebras in $\mathcal{T}$ (e.g., finite dimensional triangular local Hopf algebras).
\item
Corollary \ref{symchevprop} may be used to extend Theorem \ref{classtrhas} to arbitrary finite dimensional triangular Hopf algebras with the Chevalley property over $k$, once an extension of \cite[Corollary 6.3 \& Proposition 6.7]{g} to the supercase is available. We plan to achieve this in a future publication.
\end{enumerate}
\end{remark}

\section{Sweedler cohomology of restricted enveloping algebras and associators for their duals}

\subsection{Truncated exponentials and logarithms} Let $A$ be a finite dimensional local commutative algebra over $k$, and let $I$ be its maximal ideal. Assume that $x^p=0$ for all $x\in I$. 

\begin{proposition}\label{sertrexp}
The following hold:
\begin{enumerate}
\item
Let $n\ge 2$. Then there is a unique homomorphism of unipotent algebraic groups $$E: (I^{\otimes n},+)\to (1+I^{\otimes n},\times)$$ such that for decomposable\footnote{Here by decomposable we mean a tensor of the form $a_1\otimes\cdots\otimes a_n$, $a_i\in I$.}  $T$ we have $$E(T)=\sum_{j=0}^{p-1}\frac{T^j}{j!}.$$
\item 
Let $n\ge 2$. Then the homomorphism $E$ is an isomorphism, whose inverse $L:=E^{-1}$ satisfies
$$L(S)=\sum_{j=1}^{p-1}(-1)^{j-1}\frac{(S-1)^j}{j}$$ for every $S=E(T)$ with $T$ decomposable.
\item
Let $n\ge 3$, and let $A_i:=I^{\ot (i-1)}\ot A\ot I^{\ot (n-i)}$ for every $1\le i\le n$. Then $E$ and $L$ can be extended to homomorphisms $$E: A_1+\cdots+A_n\to (1+A_1)\cdots(1+A_n)$$
and
$$L: (1+A_1)\cdots(1+A_n)\to A_1+\cdots+A_n,$$
which are inverse to each other.
\end{enumerate} 
\end{proposition}

\begin{proof} 
(1) Recall that $I^{\otimes n}$ is the free abelian group generated by decomposable tensors $T$ modulo the relations 
$a\otimes (b_1+b_2)\otimes c=a\otimes b_1\otimes c+a\otimes b_2\otimes c$, where $a\in I^{\otimes (i-1)}$, $b_1,b_2\in I$ and $c\in I^{\otimes (n-i)}$ are decomposable, $1\le i\le n$. So our job is to show that 
$$E(a\otimes (b_1+b_2)\otimes c)=E(a\otimes b_1\otimes c)E(a\otimes b_2\otimes c).$$ 
Indeed, we have
\begin{eqnarray*}
\lefteqn{E(a\otimes (b_1+b_2)\otimes c)=\sum_{j=0}^{p-1}\frac{\left(a\otimes (b_1+b_2)\otimes c \right)^j }{j!}}\\
& = & \sum_{j=0}^{p-1}\sum_{l=0}^{j}\binom{j}{l}\frac{a^j\otimes b_1^l b_2^{j-l}\otimes c^j}{j!}= \sum_{j=0}^{p-1}\sum_{l=0}^{j}\frac{a^j\otimes b_1^lb_2^{j-l}\otimes c^j}{(j-l)!l!}\\
& = & \sum_{l=0}^{p-1}\sum_{i=0}^{p-1-l}\frac{a^{i+l}\otimes b_1^lb_2^{i}\otimes c^{i+l}}{i!l!} = \sum_{l=0}^{p-1}\sum_{i=0}^{p-1}\frac{\left(a\otimes b_1\otimes c \right)^l}{l!}\frac{\left(a\otimes b_2\otimes c \right)^i}{i!}\\
& = & E(a\otimes b_1\otimes c)E(a\otimes b_2\otimes c),
\end{eqnarray*}
as desired. (Note that the equation before the last one is justified by our assumption that $x^p=0$ for all $x\in I$.)

(2) We have to show that $$\sum_{j=1}^{p-1}(-1)^{j-1}\frac{(E(T)-1)^j}{j}=T$$ for every decomposable $T$. To prove this, it suffices to check it in the ring $k[T]/(T^p)$. But for this it is enough to check this identity in $\mathbb{Z}[1/(p-1)!][T]/(T^p)$ and then specialize to $k[T]/(T^p)$ by modding out by $p$. But for this, in turn, it suffices to establish the identity in $\mathbb{Q}[T]/(T^p)$ (as $\mathbb{Z}[1/(p-1)!][T]/(T^p)\subset \mathbb{Q}[T]/(T^p)$). Finally, in $\mathbb{Q}[T]/(T^p)$ this identity follows by truncation of the corresponding identity in $\mathbb{Q}[[T]]$ for the usual (non-truncated) $\text{Exp}$ and $\text{Log}$.

Finally, it is clear that $E$ is an isomorphism since $\gr(E)=\id$.

(3) Since $A=k\oplus I$, we can identify $A_i$ with $I^{\ot (n-1)}\oplus I^{\ot n}$ for every $1\le i\le n$, and hence extend $E$ to $A_i$ in an obvious way (as $n\ge 3$). It is easy to see that we obtain an isomorphism $E: A_i\to 1+A_i$, which means that we have an isomorphism $E: A_1+\cdots+A_n\to (1+A_1)\cdots(1+A_n)$, as desired. The proof for $L$ is similar.
\end{proof} 

\begin{remark}\label{en=1}
\begin{enumerate}
\item
Proposition \ref{sertrexp} is false for $n=1$, as $E(a+b)\ne E(a)E(b)$ if $a,b\in I$ and $E$ is defined by the above formula. Moreover the groups $I$ and $1+I$ are in general not isomorphic. For example, if $p=2$ and $A=k[x]/x^3$ then $I=k^2$ with usual addition, while $1+I=k^2$ with composition $$(a_1,b_1)*(a_2,b_2)=(a_1+a_2,b_1+b_2+a_1a_2),$$ which are not isomorphic.

Similarly, the assumption $n\ge 3$ is crucial in the extension of $E$ to $A_1+\cdots+A_n$.
\item
If $T$ is not decomposable then $E(T)$ is not given by the above formula, in general. For example, if $p=2$ then one has $E(T)=1+T$ for decomposable $T$, however $$E(T+U)=(1+T)(1+U)=1+T+U+TU,$$ which is in general not equal to $1+T+U$ (where $T,U$ are decomposable).
\end{enumerate}
\end{remark} 

\begin{corollary} 
Let $V$ be a vector space over $k$, and let $A$ be the Hopf algebra $SV/V^p$. Then we have a natural group homomorphism  
$E: V^{\otimes n}\to Z^n(A^*,\mathbb{G}_m)$ for every $n>1$, where $Z^n$ is the space of $n$-cocycles.
\end{corollary}

\begin{proof}
Since the maximal ideal $I$ of $A$ is generated by $V$, we have $I^p=0$. Thus by Proposition \ref{sertrexp}, we have a  group homomorphism $E: V^{\otimes n}\to 1+I^{\otimes n}$ for every $n>1$. 
Since $V=P(A)$, it follows that every element $v$ in $V^{\otimes n}$, viewed as an element in $\Hom((A^{\ot n})^*,k)$, is a Hochschild $n$-cocycle of $A^*$ with coefficients in $k$. Thus, using Proposition \ref{sertrexp}, it is straightforward to verify that $E(v)$ belongs to $Z^n(A^*,\mathbb{G}_m)$ for every $n>1$, as desired.
\end{proof}

\subsection{Sweedler cohomology of $u(\g)$}
We will now use Proposition \ref{sertrexp} to compute the Sweedler cohomology \cite{s2} of restricted enveloping algebras.

Let $\g$ be a finite dimensional restricted $p$-Lie algebra over $k$, and let $\Gamma$ be the finite group scheme over $k$ such that $\mathcal{O}(\Gamma)=u(\g)^*$. Then $\mathcal{O}(\Gamma)$ is a local commutative algebra such that $x^p=0$ for every $x$ in its augmentation ideal $I$.

Recall that by definition, we have $$H^{n}(\Gamma,\mathbb{G}_m)=H_{\text{Sw}}^{n}(u(\g)),$$ where $H_{\text{Sw}}^{n}(u(\g)):=H_{\text{Sw}}^{n}(u(\g),k)$ is the $n$-th Sweedler cohomology group of the cocommutative Hopf algebra $u(\g)=k\Gamma$ with coefficients in the trivial $u(\g)$-module algebra $k$ \cite{s2}. Thus, the following result gives the group $H_{\text{Sw}}^{n}(u(\g))$ in terms of the much better understood group $H^n(\Gamma,\mathbb{G}_a)$.

Recall \cite[Theorem 4.3]{s2} that the Sweedler cohomology of usual and normalized cochains in positive degree is the same. Thus, in computing Sweedler cohomology and Hochschild cohomology, we may use only normalized cochains.

\begin{theorem}\label{ncocy}
Let $\Gamma$ be as above. Then the following hold:
\begin{enumerate}
\item
$H^0(\Gamma,\mathbb{G}_m)=k^{\times}$ and $H^1(\Gamma,\mathbb{G}_m)=\Hom(\Gamma,\mathbb{G}_m)$.
\item
For every $n\ge 2$, the assignment 
\begin{equation*}
E:Z^n(\Gamma,\mathbb{G}_a)\to Z^n(\Gamma,\mathbb{G}_m),\,\,\phi\mapsto E(\phi),
\end{equation*}
is a group isomorphism, whose inverse is given by
\begin{equation*}
L:Z^n(\Gamma,\mathbb{G}_m)\to Z^n(\Gamma,\mathbb{G}_a),\,\,\Phi\mapsto L(\Phi).
\end{equation*}
\item
For every $n\ge 3$, the assignment 
\begin{equation*}
{\bf E}:H^n(\Gamma,\mathbb{G}_a)\to H^n(\Gamma,\mathbb{G}_m),\,\,[\phi]\mapsto [E(\phi)],
\end{equation*}
is a well defined group isomorphism, whose inverse is given by
\begin{equation*}
{\bf L}:H^n(\Gamma,\mathbb{G}_m)\to H^n(\Gamma,\mathbb{G}_a),\,\,[\Phi]\mapsto [L(\Phi)].
\end{equation*}
\end{enumerate}
\end{theorem}

\begin{proof}
(1) Follows from the definitions.

(2) For simplicity we will prove the claim for $n=2$, the proof for $n>2$ being similar. Let $\phi$ be an element in $Z^2(\Gamma,\mathbb{G}_a)$. Since $x^p=0$ for every $x$ in the augmentation ideal $I$ of $\mathcal{O}(\Gamma)$ it follows from Proposition \ref{sertrexp} that $\Phi:=E(\phi)$ is a well-defined invertible element of $\mathcal{O}(\Gamma)^{\ot 2}$. We have to show that $\Phi$ belongs to $Z^2(\Gamma,\mathbb{G}_m)$.

Indeed, since $$(\id \ot \Delta)(\phi)+1\ot \phi=(\Delta \ot \id)(\phi)+\phi \ot 1,$$ we have 
$$
E((\id \ot \Delta)(\phi)+1\ot \phi)=E((\Delta \ot \id)(\phi)+\phi \ot 1),
.$$
Hence by Proposition \ref{sertrexp}, we have 
\begin{equation}\label{help}
E((\id \ot \Delta)(\phi))E(1\ot \phi)=E((\Delta \ot \id)(\phi))E(\phi \ot 1).
\end{equation}
Now it is straightforward to verify that $$E(1\ot \phi)=1\ot \Phi,\,\,\,E(\phi \ot 1)=\Phi\ot 1,$$ and $$E((\id \ot \Delta)(\phi))=(\id \ot \Delta)(\Phi),\,\,\,E((\Delta \ot \id)(\phi))=(\Delta \ot \id)(\Phi).$$
Therefore, (\ref{help}) is equivalent to
$$
(\id \ot \Delta)(\Phi)(1\ot \Phi)=(\Delta \ot \id)(\Phi)(\Phi\ot 1)
.$$
Thus $\Phi$ belongs to $Z^2(\Gamma,\mathbb{G}_m)$, as desired.

The proof about $L$ is similar.

Finally, $E$ and $L$ are inverse to each other by Proposition \ref{sertrexp}.

(3) Let $\phi$ be an element in $Z^n(\Gamma,\mathbb{G}_a)$, and let $\Phi:=E(\phi)$. We have to show that $E(df)=d(E(f))$ for every $f\in I^{\ot (n-1)}$. Then we will have 
\begin{equation}\label{1more}
E(\phi +df)=E(\phi) E(df)=\Phi d(E(f)),
\end{equation}
as desired.

Indeed, (\ref{1more}) follows from our assumption that $n-1> 1$ and the commutativity of $\mathcal{O}(\Gamma)$. For simplicity we will prove it for $n=3$, the proof for $n>3$ being similar. 

So let $f\in I^{\ot 2}$, and set $F:=E(f)$. Then, using Proposition \ref{sertrexp}, we have 
\begin{eqnarray*}
\lefteqn{E(df)=E\left((\id \ot \Delta)(f)+1\ot f-(\Delta \ot \id)(f)-f \ot 1\right)}\\
& = & (\id \ot \Delta)(F)(1\ot F)(\Delta \ot \id)(F)^{-1}(F^{-1} \ot 1)=d(F),
\end{eqnarray*} 
as required.

Thus, ${\bf E}$ is a well-defined group homomorphism, as desired.

The proof that ${\bf L}$ is a well-defined group homomorphism is similar.

Finally, ${\bf E}$ and ${\bf L}$ are inverse to each other by Proposition \ref{sertrexp}. 
\end{proof}

\begin{remark}
\begin{enumerate}
\item 
The proof of Theorem \ref{ncocy}(2),(3) is similar to Sweedler's proof for usual enveloping algebras. In fact, in characteristic $p>0$ Sweedler uses truncated exponentials \cite[Theorem 4.3]{s2}.
\item
The proof of Theorem \ref{ncocy} works for any Hopf algebra quotient $H$ of $U(\g)$. 
\item
Note that $H^1(\Gamma,\mathbb{G}_m)\ncong H^1(\Gamma,\mathbb{G}_a)$.
\item
Theorem \ref{ncocy}(3) fails for $n=2$ since it may happen that
$E(df)\ne d(E(f))$ for $f\in I$ (see Remarks \ref{en=1}, \ref{zpr}).
\end{enumerate}
\end{remark}

Let $G:=\mathbb{Z}/p\mathbb{Z}$.
Recall that $\mu_p=G^D$, i.e., $\mathcal{O}(\mu_p)=\mathcal{O}(G)^*=kG$. We have $\mathcal{O}(G)=u(\g)$, where $\g$ is the abelian restricted $p$-Lie algebra with basis $x$ over $k$ such that $x^{[p]}=x$. Hence, Theorem \ref{ncocy} applies to $\Gamma:=G^D=\mu_p$.

\begin{corollary}\label{ncocymup}
Let $B:=\mathcal{O}(\mathbb{Z}/p\mathbb{Z})$ be the Hopf algebra of functions on $\mathbb{Z}/p\mathbb{Z}$ with values in $k$. Then the Sweedler cohomology of $B$ is as follows: $H^1_{\rm Sw}(B)=\mathbb{Z}/p\mathbb{Z}$, and 
$H_{\text{Sw}}^i(B)=0$ for every $i\ge 2$. 
\end{corollary}

\begin{proof}
Since $(\mathbb{Z}/p\mathbb{Z})^D$ is connected, $H^1_{\rm Sw}(B)=\Hom((\mathbb{Z}/p\mathbb{Z})^D,\mathbb{G}_m)=\mathbb{Z}/p\mathbb{Z}$.

Since $H_{\text{Sw}}^2(B)$ is the group of gauge equivalence classes of twists for $\mathbb{Z}/p\mathbb{Z}$, the claim that $H_{\text{Sw}}^2(B)=0$ follows from \cite[Corollary 6.9]{g}.

Finally, by Theorem \ref{ncocy}, $H^i(\mu_p,\mathbb{G}_a)\cong H^i(\mu_p,\mathbb{G}_m)$ for every $i\ge 3$. But it is well known that $H^i(\mu_p,\mathbb{G}_a)=0$ for every $i\ge 1$ (since $\mu_p$ is a diagonalizable group scheme).
\end{proof}

\begin{proposition}\label{ncocympgr} 
Let $G$ be a finite abelian $p$-group and $B:=\mathcal{O}(G)$ the Hopf algebra of functions on $G$ with values in $k$. Then the Sweedler cohomology of $B$ is as follows: $H^1_{\rm Sw}(B)=G$ and $H^i_{\rm Sw}(B)=0$ for $i\ge 2$. 
\end{proposition} 

\begin{proof} 
The first statement is clear, since $H^1_{\rm Sw}(B)=\Hom(G^D,\mathbb{G}_m)=G$ (as $G^D$ is connected). 

To prove the second statement, we interpret $H^i_{\rm Sw}(B)$ as $H^i(G^D,\mathbb{G}_m)$. We have $|G|=p^n$. The proof is by induction in $n$, starting with the base case $n=1$, which has already been settled in Corollary \ref{ncocymup}. 

Let $K$ be a subgroup of index $p$ in $G$, and consider the Lyndon-Hochschild-Serre
sequence for the cohomology 
$H^n(G^D,\mathbb{G}_m)$ attached to the short exact sequence 
$$
1\to (G/K)^D\to G^D\to K^D\to 1.
$$
The $E_2$ page of this spectral sequence consists of the groups $$H^r(K^D,H^q((G/K)^D,\mathbb{G}_m)).$$ So it suffices to show that these groups vanish 
when $r+q=i\ge 2$. To this end, note that $H^q((G/K)^D,\mathbb{G}_m)=H^q(\mu_p,\mathbb{G}_m)=0$ for $q\ge 2$ 
by Corollary \ref{ncocymup}. Thus it remains to show that $$H^i(K^D,H^0((G/K)^D,\mathbb{G}_m))=0\,\,\,\text{and}\,\,\,H^{i-1}(K^D,H^1((G/K)^D,\mathbb{G}_m))=0.$$ The first group equals $H^i(K^D,\mathbb{G}_m)$, which vanishes 
by the induction assumption. The second group is $H^{i-1}(K^D,G/K)=H^{i-1}(K^D,\mathbb{Z}/p\mathbb{Z})$, which vanishes for $i\ge 2$ since 
$K^D$ is connected and $G/K$ is discrete. 

The proof of the proposition is complete. 
\end{proof} 

\begin{corollary}\label{newpgrp}
The group algebra $kG$ of any finite abelian $p$-group $G$ does not admit non-trivial twists or associators. \qed
\end{corollary}  

\begin{remark}\label{zpr}
Proposition \ref{ncocympgr} was proved by Guillot for $G:=(\mathbb{Z}/2\mathbb{Z})^r$, $r\ge 1$, and $p=2$ \cite[Theorem 1.1]{gu}. Note that, for $\mathbb{Z}/2\mathbb{Z}$ with generator $g$, we have $$E(d(1+g))=1+(1+g)\ot (1+g)\ne 1=d(E(1+g)).$$
\end{remark}

Let $G:=\prod_{j=1}^n \alpha_{p,r_j}$. Then $\mathcal{O}(G)=u(\g)$, where $\g$ is an $\sum_{j=1}^n r_j$-dimensional abelian restricted $p$-Lie algebra over $k$ equipped with the $p$-associative power as the $p$-operator. Hence, Theorem \ref{ncocy} applies to $\Gamma:=G^D$.

\begin{corollary}\label{swexgd}
Let $G:=\prod_{j=1}^n \alpha_{p,r_j}$. Then 
for every $i\ge 3$, we have $$H_{\text{Sw}}^i (\mathcal{O}(G))\cong H^i (G^D,\mathbb{G}_a ).$$ 
More explicitly, let $\mathfrak{a}:=\text{Lie}(G^D)$ and let $\widetilde{\mathfrak a}$ be the filtered space associated to $G^D$ as in Proposition \ref{kunneth}. Then the following hold for every $i\ge 3$:
\begin{enumerate}
\item
If $p>2$, then 
$H_{\text{Sw}}^i(\mathcal{O}(G))$ is isomorphic to the $i$-th component of the graded supercommutative algebra $\wedge\mathfrak{a}\otimes S'(\widetilde{\mathfrak{a}})$ equipped with its standard grading (see Proposition \ref{kunneth}(1)).
\item
If $p=2$, then $H_{\text{Sw}}^i (\mathcal{O}(G))$ is isomorphic to the $i$-th component of the graded commutative algebra $S\mathfrak{a}\otimes S'(\widetilde{\mathfrak{a}})/(z^2=\phi(z),z\in \mathfrak{a})$ equipped with its standard grading (see Proposition \ref{kunneth}(2)).
\end{enumerate}
\end{corollary}

\begin{proof}
Follows from the above and Proposition \ref{kunneth}.
\end{proof}

\begin{proposition}\label{h2sw}
Let $\g$ be a finite dimensional commutative $p$-Lie algebra over $k$, equipped with the zero $p$-power map. Then   
$H^2_{\text{Sw}}(u(\g))\cong\wedge^2\g^*$. In particular, $H^2_{\text{Sw}}(u(\g))\ncong H^2(u(\g),k)$ for $\g\ne 0$.
\end{proposition}

\begin{proof}
By \cite[Proposition 6.8]{g}, computation of $H^2_{\text{Sw}}(u(\g))$ is equivalent to the classification of twists for $u(\g)^*$ up to gauge equivalence. Let $J$ be such a twist. If $J\ne 1$, let $r$ be the lowest degree part of $J-1$ (say, it has degree $n$). Then $r$ is a $2$-cocycle of $u(\g)$ with coefficients in $k$. Recall that $H^2(u(\g),k)=\wedge^2\g^*\oplus (\g^*)^{(1)}$ and the map $\xi$ (see Subsection 2.6). Thus $r$ can be written as   
$$r=\Delta(x)-x\ot 1-1\ot x+s+\xi(y),$$ 
where $x\in u(\g)^*[n]$, $s\in \wedge^2\g^*$, $y\in \g^*$, and $s$ can be nonzero only for $n=2$, while $y$ can be nonzero only for $n=p$. Thus by gauge transformation by $(1+x)E(y)$, where $E(y)=\sum_{m=0}^{p-1}\frac{y^m}{m!}$, we can bring $J$ to a form in which the degree $n$ part of $J-1$ is $s$. Hence $J':=JE(-\widetilde{s})$, where $\widetilde{s}$ is a preimage of $s$ in $\g^*\ot \g^*$, is a twist such that $J'-1$ has degree at least $n+1$. (For the statement that $E(-\widetilde{s})$ is a twist see also \cite[Example 6.13]{g}.)

This argument shows that any twist $J$ is gauge equivalent to $E(\widetilde{s})$ for some $\widetilde{s}$. Also, by looking at the degree $2$ part, it is clear that $E(\widetilde{s})$ is equivalent to $E(\widetilde{s'})$ if and only if $s=s'$, so the assignment $s\mapsto E(\widetilde{s})$ is an isomorphism between $\wedge^2\g^*$ and $H^2_{\text{Sw}}(u(\g))$, as desired.
\end{proof}

\begin{remark}
By \cite[Corollary 6.10]{g}, the twist $E(\widetilde{s})$ is minimal if and only if $s$ is non-degenerate.
\end{remark}

\subsection{Associators for $u(\g)^*$} Let $G$ be a finite commutative group scheme over $k$. 
Then the following lemma follows from the definitions in a straightforward manner.

\begin{lemma}\label{assoccomm}
There is a bijection between $H^3(G^D,\mathbb{G}_m)$ and the group of pseudotwist equivalence classes of associators for $kG=\mathcal{O}(G^D)$. \qed
\end{lemma}

Thus, by Lemma \ref{assoccomm}, every normalized $3$-cocycle $\Phi$ in $Z^3(G^D,\mathbb{G}_m)$ can be viewed as an associator for $kG$. We will denote the associated (cocommutative and commutative) quasi-Hopf algebra by $(kG,\Phi)$. Note that $(kG,\Phi)$ is pseudotwist equivalent to a Hopf algebra (i.e., to $kG$) if and only if $\Phi$ is a coboundary.

\begin{corollary}\label{assocncocy}
Let $\g$ be a finite dimensional restricted $p$-Lie algebra over $k$, and let $\Gamma$ be the finite group scheme over $k$ such that $\mathcal{O}(\Gamma)=u(\g)^*$. Then every associator for $u(\g)^*$ is of the form $E(\phi)$ for a unique $\phi\in Z^3(\Gamma,\mathbb{G}_a)$.
\end{corollary}

\begin{proof}
Follows immediately from Theorem \ref{ncocy}.
\end{proof}

\begin{example}\label{diag}
By Corollary \ref{newpgrp}, 
the group algebra $kG$ of any finite abelian $p$-group $G$ does not admit non-trivial associators.
\end{example}

\begin{example}\label{exgd}
Let $G:=\prod_{i=1}^n \alpha_{p,r_i}$. By Corollary \ref{swexgd}, every associator for $kG$ is of the form  
$E(\phi)$ for a unique $\phi\in Z^3(G^D,\mathbb{G}_a)$. 
In particular, every associator for $k\alpha_p^n$ is of the form $E(\phi)$ for a unique $\phi\in Z^3(\alpha_p^n,\mathbb{G}_a)$ since $(\alpha_p^n)^D=\alpha_p^n$.
\end{example}

Note that Proposition \ref{hochcoh} and Example \ref{exgd} (together with the K\"{u}nneth formula) imply an explicit classification of associators for $k[\prod_{i=1}^n \alpha_{p,r_i}]$. In the next theorem we illustrate it with the simplest example.

Let 
\begin{equation}\label{phis}
\Phi:=1+
\sum_{i=1}^{p-1}\frac{1}{i}{p-1\choose i-1}x\ot x^i\ot x^{p-i}\in k\alpha_p^{\ot 3}.
\end{equation}

\begin{theorem}\label{assocalphap}
If $\Phi'$ is a non-trivial associator for $k\alpha_p$ then $(k\alpha_p,\Phi')$ is pseudotwist equivalent to $(k\alpha_p,\Phi)$.
In particular, $(k\alpha_p,\Phi)$ is a local (cocommutative and commutative ) quasi-Hopf algebra of dimension $p$ over $k$, which is not pseudotwist equivalent to a Hopf algebra.
\end{theorem}

\begin{proof}
Let $\mathcal{O}(\alpha_p)=k[x]/(x^p)$. By (\ref{h33}), $H^3(\alpha_p,k)\cong \mathfrak{a}\ot \mathfrak{a}^{(1)}$ if $p>2$ (as $\wedge^3 \mathfrak{a} =0$, since $\mathfrak{a}$ is $1$-dimensional), 
and by (\ref{h32}), $H^3(\alpha_2,k)\cong S^3 \mathfrak{a}$. In particular, since $\mathfrak{a}$ is $1$-dimensional, we see that $H^3(\alpha_p,k)$ is $1$-dimensional for every $p$. It follows that the class $[\phi]$ forms a basis for $H^3(\alpha_p,k)$, where $$\phi:=x\ot \overline{\xi}(x)=\sum_{i=1}^{p-1}\frac{1}{i}{p-1\choose i-1}x\ot x^i\ot x^{p-i}.$$

Now, since $(x^i\ot x^{p-i})(x^j\ot x^{p-j})=0$ for every $1\le i,j\le p-1$, it follows from Theorem \ref{ncocy} that  
\begin{eqnarray*}
\lefteqn{
\Phi_s:=E(s(x\ot \overline{\xi}(x)))}\\
& = & \prod_{i=1}^{p-1}E\left(\frac{s}{i}{p-1\choose i-1}x\ot x^i\ot x^{p-i}\right)=1+
\sum_{i=1}^{p-1}\left(\frac{s}{i}{p-1\choose i-1}x\ot x^i\ot x^{p-i}\right)
\end{eqnarray*}
is an associator for $k\alpha_p$ for every $s\in k$, and every associator of $k\alpha_p$ is pseudotwist equivalent to one of this form.

Finally, observe that $\Phi_0=1$, and $(k\alpha_p,\Phi)\cong (k\alpha_p,\Phi_s)$ for every $0\ne s\in k$. Indeed, the map $(k\alpha_p,\Phi)\to (k\alpha_p,\Phi_s)$ given by $x\mapsto \mu x$, where $\mu\in k$ satisfies $\mu^{p+1}=s$, is an isomorphism of quasi-Hopf algebras. 

We are done.
\end{proof}

\section{Finite unipotent tensor categories}

In this section we consider finite dimensional (not necessarily triangular) local quasi-Hopf algebras and finite (not necessarily symmetric) unipotent tensor categories over $k$.

\begin{theorem}\label{main2}
Let $H$ be a finite dimensional local quasi-Hopf algebra over $k$. Then the following hold:
\begin{enumerate}
\item
There exists a unique finite connected unipotent group scheme $U$ over $k$ that has a $\mathbb{G}_m$-action such that $\gr(H)\cong kU$ as Hopf algebras. In particular, 
$\gr(H)$ is a cocommutative local Hopf algebra generated by primitive elements, i.e., a quotient of an enveloping algebra.
\item

$\dim(H)=|U|$, hence $\dim(H)$ is a power of $p$.
\end{enumerate}
\end{theorem}

\begin{proof} 
Clearly $I:=\Rad(H)$ is a quasi-Hopf ideal of $H$, so the associated graded algebra $\gr(H)=\bigoplus_r H[r]$ is a local quasi-Hopf algebra. Since the associator of $\gr(H)$ lives in $H[0]^{\otimes 3}$, it is equal to $1$. Thus, $\gr(H)$ is a radically graded local {\em Hopf algebra}. Since $H[1]$ consists of primitive elements and by \ref{sta}, $H[1]$ generates $\gr(H)$, it follows that $\gr(H)$ is cocommutative. Therefore $\gr(H)\cong kU$ is the group algebra of some finite connected (as $\gr(H)^*\cong \mathcal{O}(U)$ has a unique grouplike element, by locality of $kU$) unipotent group scheme $U$.
\end{proof}

\begin{theorem}\label{uniptc}
Let $\mathcal{C}$ be a finite unipotent tensor category over $k$.
Then the following hold:
\begin{enumerate}
\item
$\mathcal{C}$ is tensor equivalent to the representation category $\Rep(H)$ of a finite dimensional local quasi-Hopf algebra $H$ over $k$.
\item

$\FPdim(\mathcal{C})$ is a power of $p$.
\end{enumerate}
\end{theorem}

\begin{proof} 
(1) Since the unit object is the unique simple object in $\mathcal{C}$, $\mathcal{C}$ is integral. Therefore, by \cite[Proposition 2.6]{eo}, $\mathcal{C}$ is tensor equivalent to the representation category $\Rep(H)$ of a finite dimensional quasi-Hopf algebra $H$ over $k$, and since $\mathcal{C}$ is unipotent, $H$ is local.

(2) Follows from Part (1) and Theorem \ref{main2}.
\end{proof}

Recall the associator $\Phi$ for $k\alpha_p$ given in (\ref{phis}).

\begin{theorem}\label{cor2} 
Let $H$ be a local quasi-Hopf algebra over $k$ of dimension $p$. Then $H$ is pseudotwist equivalent to either the Hopf algebra $k[\mathbb{Z}/p\mathbb{Z}]$, the Hopf algebra $k\alpha_p$, or the quasi-Hopf algebra $(k\alpha_p,\Phi)$.
\end{theorem}

\begin{proof} 
By Theorem \ref{main2}, $\gr(H)=kU$ for a connected unipotent group scheme $U$ of order $p$. Thus  
$U=\alpha_p$, and we can identify $\gr(H)$ with $k[x]/(x^p)$ as graded vector spaces. Since $H$ is generated by $x$ in degree $1$, it is clear that $H=k[x]/(x^p)$ as {\em algebras}. Thus, $H$ is a local {\em commutative Hopf algebra}, equipped with an invertible element in $H^{\ot 3}$ satisfying (\ref{phi1}), (\ref{phieps}) (as (\ref{phi4}) is satisfied automatically).

Now, if $H$ has $p$ distinct grouplike elements then $H=k[\mathbb{Z}/p\mathbb{Z}]$, so the claim follows from Example \ref{diag}. 

Otherwise, $H$ is connected, hence must have a non-trivial primitive element, which implies that $H=k[x]/(x^p)$ as {\em Hopf algebras}, so the claim follows from Theorem \ref{assocalphap}.
\end{proof}

\begin{corollary}\label{cor3}
Let $\mathcal{C}$ be a unipotent tensor category over $k$ of FP-dimension $p$. Then $\mathcal{C}$ is tensor equivalent to one of the following three non-equivalent tensor categories: $\Rep(\mathbb{Z}/p\mathbb{Z})$, $\Rep(\alpha_p)$, or $\Rep(\alpha_p,\Phi)$.
\end{corollary}

\begin{proof}
Follows from Theorems \ref{uniptc}, \ref{cor2}.
\end{proof}

\begin{remark}\label{newrmk}
Connected Hopf algebras of dimension $p^2$ and $p^3$ over $k$ were classified in \cite{wan,nww}, and those with large abelian primitive spaces were classified in \cite{wan2}. Also, it was recently proved in \cite{nw} that the only Hopf algebras of dimension $p$ over $k$ are $k[\mathbb{Z}/p\mathbb{Z}]$, $\mathcal{O}(\mathbb{Z}/p\mathbb{Z})$ and $k\alpha_p$.
\end{remark}

\section{Appendix: Maximal Tannakian and super-Tannakian subcategories of symmetric tensor categories}

\vskip .1in
\centerline{\bf By Kevin Coulembier and Pavel Etingof}
\vskip .1in

Using cotriangular coquasi-Hopf algebras, Theorem \ref{classchptr} and Corollary \ref{classuniptr} can be extended to the case when the category $\C$ is not necessarily finite, and it is only required that the semisimple part of $\C$ is an integral fusion category. 
Namely, we have the following theorem. 

\begin{theorem}\label{gene} Let $\C$ be a symmetric tensor category over an algebraically closed field $k$ of characteristic $p>2$ with the Chevalley property, such that its semisimple part $\C_0$ is an integral fusion category. Then $\C$ admits a symmetric fiber functor to the category $\sVec$ of supervector spaces. Thus, there exists a unique affine supergroup scheme $\mathcal{G}$ over $k$ and a grouplike element $\varepsilon \in k\mathcal{G}$ of order $\le 2$, whose action by conjugation on $\mathcal{O}(\mathcal{G})$ is the parity automorphism, such that $\mathcal{C}$ is symmetric tensor equivalent to $\Rep(\mathcal{G},\varepsilon)$. In particular, if $\C$ is unipotent then there is a unipotent affine group scheme $\mathcal{G}$ such that $\C\cong \Rep(\mathcal{G})$. 
\end{theorem} 

This theorem also follows from the combination of \cite[Theorem 8.1]{eov}, \cite[Corollary 3.5]{Et} and \cite[Proposition 6.2.2]{Co}.

Since it is of independent interest, we will prove the following proposition, which is an improvement of \cite[Proposition 6.2.2]{Co}, and from which we will derive Theorem~\ref{gene}. Furthermore, we make the proof more self-contained by providing an alternative for the use of \cite[Corollary 3.5]{Et}. 

From now on, we let $k$ be an algebraically closed field. Recall that a symmetric tensor category over $k$ is called (super-)Tannakian if it admits a tensor functor to the category of finite dimensional (super)vector spaces over $k$. 

\begin{proposition}\label{622modif}
The following hold: 
\begin{enumerate}
\item Let $\C$ be a symmetric tensor category over $k$. Then $\C$ has a unique maximal Tannakian subcategory $\C_T^+$ and, if ${\rm char}(k)\ne 2$, a unique maximal super-Tannakian subcategory $\C_T$.  
\item
If ${\rm char}(k)\ne 2$ then $\C_T,\C_T^+$ are Serre subcategories of $\C$.
\end{enumerate} 
\end{proposition} 
The proposition will be proved in Subsection \ref{secproof}.

\begin{remark} Note that Proposition \ref{622modif}(2) fails for ${\rm char}(k)=2$ (for $\C_T^+$), see Remark \ref{venk} above and \cite[Subsection 1.5]{v}. 
\end{remark} 

\noindent {\bf Proof of Theorem \ref{gene}}. As explained in Section 4 above, by \cite[Theorem 8.1]{eov}, the subcategory $\C_0\subset \C$ is super-Tannakian. Thus the maximal super-Tannakian subcategory $\C_T$ of $\C$ contains all
simple objects of $\C$ and is a Serre subcategory by Proposition \ref{622modif}(ii), which implies that $\C_T=\C$, i.e., $\C$ is super-Tannakian, as claimed.

\subsection{Extending quasi-fiber functors}\label{quasifib} 
Let $\cA$ be an artinian category over $k$, see \cite[\S 1.8]{egno}. It follows easily that each simple object in $\cA$ has a projective cover in the pro-completion $\Pro\cA$. We denote by $\cP(\cA)$ the full subcategory of $\Pro\cA$ of products of such projective covers in which each cover appears only finitely many times. In this case, the product is also a categorical coproduct and therefore objects of $\cP(\cA)$ are projective. 

\begin{lemma}\label{LemProj}
The following hold:
\begin{enumerate}
\item The assignment $P\mapsto \Hom(P,-)$ yields an equivalence between $\cP(\cA)^{\op}$ and the category of $k$-linear exact functors $\cA\to\Vecc$.
\item Let $\cV$ be semisimple artinian. For a $k$-linear exact functor $F:\cA\to\cV$ we have the induced homomorphism $\Gr(F):\Gr(\cA)\to\Gr(\cV)$ between the Grothendieck groups. The map $F\mapsto \Gr(F)$ is surjective onto the set of group homomorphisms  $\Gr(\cA)\to\Gr(\cV)$ which have non-negative coefficients with respect to the bases labelled by simple objects. The fibers of the map are precisely the isomorphism classes of functors.
\item Let $\cA$ and $\cV$ be as in (2), with $\cA_0$ an abelian subcategory of $\cA$ containing all simple objects. Restriction is a dense (essentially surjective) functor from the category of $k$-linear exact functors $\cA\to\cV$ to the corresponding category of functors $\cA_0\to\cV$. Furthermore, when we consider the functor categories with same objects but only isomorphisms, this restriction functor is full.
\end{enumerate}
\end{lemma}

\begin{proof}
Consider an exact linear functor $\cA\to\Vecc$.
It follows from the artinian nature of $\cA$ that the functor is representable by an object in $\cP(\cA)$. Part (1) thus follows from the Yoneda lemma. 

We prove part (2) for the special case $\cV=\Vecc$, the general case then follows from this. By Part (1) it suffices to observe that $P\mapsto \dim\Hom(P,-)$ can be interpreted as a bijection between the set of isomorphism classes of objects in $\cP(\cA)$ and maps from the set of simple objects in $\cC$ to $\mathbb{Z}_+$.

Finally we prove part (3). For $P$ in $\cP(\cA)$ we denote by $P_0$ its maximal quotient contained in $\Pro\cA_0$. In particular, $P$ is the projective cover of $P_0$ in $\Pro\cA$. Restriction of exact linear functors from $\cA$ to $\cA_0$ corresponds, using the equivalence in part (1), to mapping $\Hom(P,-)$ to $\Hom(P_0,-)$. It follows from the defining properties of projective objects that the functor $P\mapsto P_0$ is full and dense between the respective categories of projective objects and isomorphisms.
\end{proof}

For functors $F:\cC\to\cC'$ between tensor categories, we will be interested in isomorphisms
\begin{equation}\label{eqJ}J_{XY}: F(X)\otimes F(Y)\stackrel{\sim}{\to} F(X\otimes Y),\qquad\mbox{for all $X,Y\in\cC$,}
\end{equation}
natural in both variables.
Following \cite[Definition~5.1.1]{egno}, for tensor categories $\cC,\cV$ with $\cV$ semisimple, a quasi-tensor functor $F:\cC\to\cV$ is a $k$-linear, exact and faithful functor with $F(\mathbf{1})\cong\mathbf{1}$, equipped with $J$ as in \eqref{eqJ}. 

\begin{proposition}
\label{PropExtend}
Consider a tensor category $\cC$ with a tensor subcategory $\cC_0$ containing all simple objects of $\cC$, and a semisimple tensor category $\cV$. Then each quasi-tensor functor $(F_0,J_0)$ from $\cC_0$ to $\cV$ extends to a quasi-tensor functor $(F,J)$ from $\cC$ to $\cV$.
\end{proposition}

\begin{proof}
That $F_0$ extends to a linear exact functor $F$ follows from Lemma~\ref{LemProj}(3). Faithfulness and $F(\mathbf{1})\cong\mathbf{1}$ are automatically inherited from $F_0$.

Now we will use Deligne's tensor product $\cA\boxtimes\cA'$ of artinian categories $\cA$ and $\cA'$, see e.g. \cite[\S 1.11]{egno}. By construction, $\cA\boxtimes\cA'$ is again an artinian category and we have a bilinear bifunctor $\cA\times\cA'\to\cA\boxtimes\cA'$ which induces equivalences between the category of exact $k$-linear functors $\cA\boxtimes\cA'\to\cV$ and the category of biexact bilinear bifunctors $\cA\times\cA'\to\cV$. It thus follows that we can identify $J$ as in \eqref{eqJ} with a natural isomorphism between the two functors $\cC\boxtimes\cC\to \cV$ which correspond to
$$\cC\times\cC\xrightarrow{F\times F}\cV\times\cV\xrightarrow{\otimes}\cV\quad\mbox{and}\quad \cC\times\cC\xrightarrow{\otimes}\cC\xrightarrow{F}\cV.$$
Applying this to $J_0$ yields an isomorphism between the two functors on $\cC_0\boxtimes\cC_0$ induced from $F_0$. By construction, both functors extend to functors on $\cC\boxtimes\cC$ which are induced from $F$. That the isomorphism corresponding to $J_0$ lifts to the functors on $\cC\boxtimes\cC$ follows again from Lemma~\ref{LemProj}(3).
\end{proof}

\begin{remark}Consider tensor categories $\cC,\cV$, with $\cV$ semisimple. For an exact functor $F:\cC\to\cV$ there exists $J$ as in \eqref{eqJ} if and only if $\Gr(F):\Gr(\cC)\to\Gr(\cV)$ is a ring homomorphism. This can be proved using the techniques from the proof of Proposition~\ref{PropExtend}. 
\end{remark}

Recall the notion of locally semisimple tensor categories from \cite[Section 3]{Co}.

\begin{corollary}\label{CorLocSS}
Assume char$(k)\not=2$ and let $\cC$ be a tensor category over $k$ with a Tannakian or super-Tannakian subcategory $\cC_0$ which contains all simple objects. Then $\cC$ is locally semisimple.
\end{corollary}

\begin{proof}
Consider the setting of Proposition~\ref{PropExtend} with $\cV=\Vecc$ or $\cV=\sVec$.
It suffices to show that when $(F_0,J_0)$ is a tensor functor, and hence a (super) fiber functor, $\cC$ is locally semisimple. Consider an arbitrary object $Y$ in $\cC$ with a Loewy filtration. Hence $\gr Y$ is in $\cC_0$. It suffices to show that the canonical epimorphism $S^n(\gr Y)\twoheadrightarrow \gr (S^nY)$ between symmetric powers is always an isomorphism.

First assume that $\cC_0$ is Tannakian. Assume the contrary, i.e., that there exists $Y\in\cC$ which does not yield an isomorphism. Since $F:\cC\to\Vecc$ from Proposition~\ref{PropExtend} is faithful, this means that
$$H_0(S_n,\gr M)\twoheadrightarrow \gr H_0(S_n, M),\qquad \mbox{with $M:=F(Y^{\otimes n})$},$$ 
is not an isomorphism. Here $M$ is a (filtered) $k[S_n]$-module through $$S_n\to\End(Y^{\otimes n})\to \End_k(M)$$ and, since $F_0$ is symmetric monoidal, the graded module $\gr M$ is isomorphic to $V^{\otimes n}$, for the vector space $V:=F(\gr Y)$. Consequently $\gr M$ is a direct sum of permutation modules. From Shapiro's lemma and Mackey's theorem, it follows that there are no first extensions between permutation modules when char$(k)\not=2$. Hence $\Ext^1_{S_n}(\gr M,\gr M)=0$, which means that $\gr M\cong M$. The latter isomorphism contradicts the fact that the dimension of their spaces of coinvariants would differ.

Now assume that $\cC_0$ is super-Tannakian. We can then proceed as in the previous paragraph to find a filtered $S_n$-representation $M$ in $\sVec$. Now $\gr M$ will be of the form $V^{\otimes n}$ for a super vector space $V$. Hence $M$ is a direct sum of permutation modules and sign-twisted permutation modules. If char$(k)\not\in\{2,3\}$ it follows as before that such representations have no self-extensions and therefore $\gr M\cong M$. The conclusion follows as in the previous paragraph.

In case $p=3$, one can calculate directly that $\Ext^1_{S_3}(V^{\otimes 3}, V^{\otimes 3})=0$, for any $V\in \sVec$, where the extension is taken in the category of super representations of $S_3$. The latter decomposes into the direct sum of two copies of the ordinary category $\Rep (S_3)$. We can therefore conclude as above that $S^3\gr Y\twoheadrightarrow \gr S^3Y$ is always an isomorphism. That this is sufficient for $\cC$ to be locally semisimple follows from \cite[Theorem~C]{Co}.
\end{proof}

\subsection{Proof of Proposition \ref{622modif}}\label{secproof}

We start with the following lemma, which appears in the Tannakian case as \cite[Proposition 1]{B} and can be proved similarly in the super-Tannakian case. 

\begin{lemma}\label{surt} Let $\D,\E$ be symmetric tensor categories over $k$ and $F: \D\to \E$ a surjective symmetric tensor functor (i.e., every object of $\E$ is a subquotient of $F(X)$ for some $X\in \D$). If $\D$ is finitely tensor-generated and (super-)Tannakian, then so is $\E$. \qed
\end{lemma}  

\begin{remark} \cite[Proposition 1]{B} addresses a more general case when $\E$ and $F$ are not necessarily symmetric, 
and assumes that $F$ is a central functor. It applies to our situation because a symmetric monoidal functor is automatically central. 
Also, in our setting we do not need \cite[Lemma 3]{B}, which is needed in \cite[Proposition 1]{B} because it is not assumed there that $\E$ is rigid. 
\end{remark}  

(1) Let $\C_T^+\subset \C$ be the full subcategory of $X\in \C$ which tensor generate a Tannakian subcategory $\C(X)\subset \C$. By definition, $\C_T^+$ is closed under subquotients. Also, if $X,Y\in \C_T^+$ then we have a surjective symmetric tensor functor \linebreak $\C(X)\boxtimes \C(Y)\to \C(X\oplus Y)$, where $\boxtimes$ is the Deligne tensor product. Hence by Lemma \ref{surt}, $\C(X\oplus Y)$ is Tannakian, i.e., $X\oplus Y\in \C_T^+$. Hence we also have $X\otimes Y\in \C_T^+$ (as 
$X\otimes Y$ is a direct summand in $(X\oplus Y)^{\otimes 2}$). Thus, $\C_T^+$ is a tensor subcategory of $\C$, in which every object generates a Tannakian subcategory. Hence by a result of Deligne (see \cite[A.2.3 and A.4.1]{Co}), $\C_T^+$ is Tannakian. By definition, 
$\C_T^+$ contains every Tannakian subcategory of $\C$, so $\C_T^+$ is the unique maximal Tannakian subcategory of $\C$.

Similarly, for ${\rm char}(k)\ne 2$ the full subcategory $\C_T\subset \C$ of $X\in \C$ which tensor generate a super-Tannakian subcategory $\C(X)\subset \C$ is the unique maximal super-Tannakian subcategory of $\C$. 

(2) Let $\C_T'$ be the smallest Serre subcategory of $\C$ containing $\C_T$. By Corollary~\ref{CorLocSS}, the tensor category $\C_T'$ is locally semisimple. Hence by \cite[Proposition 6.2.2]{Co}, the maximal super-Tannakian subcategory of $\C_T'$ (which is clearly $\C_T$) is a Serre subcategory and hence $\C_T=\C_T'$.

The proof that $\C_T^+$ is a Serre subcategory of $\C$ is similar, using \cite[Proposition 6.1.2]{Co}.  

\vskip .05in

{\bf Acknowledgements.} P. E. is grateful to S. Gelaki and V. Ostrik for useful discussions.

\end{document}